\documentclass[12pt, reqno]{amsart}
\usepackage{amssymb,amsthm,amsmath}
\usepackage[numbers,sort&compress]{natbib}
\usepackage{amssymb,amsmath}
\usepackage{amsfonts}
\usepackage{mathrsfs}
\usepackage{latexsym}
\usepackage{amssymb}
\usepackage{amsthm}
\usepackage{indentfirst}
\date{July 18, 2015}

\let\oldsection\section
\renewcommand\section{\setcounter{equation}{0}\oldsection}

\newtheorem{corollary}{Corollary}[section]
\newtheorem{theorem}{Theorem}[section]
\newtheorem{lemma}{Lemma}[section]
\newtheorem{proposition}{Proposition}[section]
\newtheorem{definition}{Definition}[section]

\newtheorem{remark}{Remark}[section]

\begin{document}

\title[A tropical atmosphere model with moisture]{A tropical atmosphere model with moisture: global well-posedness and relaxation limit}


\author{Jinkai~Li}
\address[Jinkai~Li]{Department of Computer Science and Applied Mathematics, Weizmann Institute of Science, Rehovot 76100, Israel}
\email{jklimath@gmail.com}

\author{Edriss~S.~Titi}
\address[Edriss~S.~Titi]{
Department of Mathematics, Texas A\&M University, 3368 TAMU, College Station, TX 77843-3368, USA. ALSO, Department of Computer Science and Applied Mathematics, Weizmann Institute of Science, Rehovot 76100, Israel.}
\email{titi@math.tamu.edu and edriss.titi@weizmann.ac.il}



\begin{abstract}
In this paper, we consider a nonlinear interaction system between the barotropic mode and the first baroclinic mode of
the tropical atmosphere with moisture; that was derived in
[Frierson,~D.~M.~W.; Majda,~A.~J.; Pauluis,~O.~M.: \emph{Dynamics of precipitation
fronts in the tropical atmosphere: a novel relaxation limit}, Commum. Math. Sci.,
\bf2~\rm(2004), 591--626.] We establish the global existence and uniqueness
of strong solutions to this system, with initial data in $H^1$,
for each fixed convective adjustment relaxation time parameter
$\varepsilon>0$. Moreover, if the initial data enjoy slightly more regularity than $H^1$, then the unique strong solution depends continuously on the initial data. Furthermore, by establishing several appropriate $\varepsilon$-independent
estimates, we prove that the system converges to a limiting system, as the relaxation time parameter $\varepsilon$ tends to zero, with   
convergence rate  of the order $O(\sqrt\varepsilon)$. Moreover, the limiting
system has a unique global strong solution, for any initial data in $H^1$, and such unique strong solution depends continuously on
the initial data if the the initial data posses slightly more regularity than $H^1$. Notably, this solves the
\textsc{viscous version} of an open problem proposed  in the above mentioned paper of Frierson, Majda and Pauluis.
\end{abstract}

\maketitle


\allowdisplaybreaks
\section{Introduction}
\label{sec1}
\subsection{The primitive equations for planetary atmospheric dynamics}
In the context of large-scale atmosphere, the ratio of the vertical scale to the horizontal scale is very small, which, by scale analysis, see, e.g., \cite{PED,VALLIS}, leads to the hydrostatic approximation in the vertical momentum equation. This small aspect ratio limit can be rigorously justified, see \cite{LITITIHYDRO,AZGU}.
Taking into account the Boussinesq approximation and the hydrostatic approximation to the Navier-Stokes equations, one obtains the primitive equations, which model the large-scale atmospheric dynamics.

The primitive equations read  (see, e.g., \cite{HAWI,LEWAN,MAJBOOK,PED,VALLIS,WP,ZENG})
\begin{equation}\label{PE}
  \left\{
  \begin{array}{l}
    \partial_t\textbf{V}+(\textbf{V}\cdot\nabla_h)\textbf{V}+W\partial_z\textbf{V}
     -\mu\Delta\textbf{V}+\nabla_h\Phi=0,\\
    \partial_z\Phi=\frac{g\Theta}{\theta_0},\\
    \partial_t\Theta+\textbf{V}\cdot\nabla_h\Theta+W\partial_z\Theta +\frac{N^2\theta_0}{g}W=S_{\Theta},\\
    \nabla_h\cdot\textbf{V}+\partial_zW=0,
  \end{array}
  \right.
\end{equation}
where the unknowns $\textbf{V}=(V_1,V_2)^T$, $W$, $\Phi$ and $\Theta$ are the horizontal velocity field, vertical velocity, pressure and potential temperature, respectively, while the positive constant $\mu$ is the viscosity coefficient. The total potential temperature is given by
$$
\Theta^{\text{total}}(x,y,z,t)=\theta_0+\bar\theta(z)+\Theta(x,y,z,t),
$$
where $\theta_0$ is a positive reference constant temperature and $\bar\theta$ defines the vertical profile background stratification, satisfying $N^2=(g/\theta_0)\partial_z\bar\theta>0$, where $N$ is the Brunt-V\"ais\"al\"a buoyancy frequency. Here we use $\nabla_h=(\partial_x,\partial_y)$ to denote the horizontal gradient and $\textbf{V}^\perp=(-V_2, V_1)^T$.

During the last two
decades, a lot of efforts have been done on the mathematical studies of the primitive equations.
Up to now, it has been known that the primitive equations, with full viscosity and full diffusivity, have global weak
solutions (but the uniqueness is still unclear),
see \cite{LTW92A,LTW92B,LTW95}, and have a unique
global strong solution, see \cite{CAOTITI07,KUZI1,KUZI2,KOB}, and also see \cite{CAOTITI12,CAOLITITI1,CAOLITITI2,LITITITCM}
for some recent developments
towards the direction of partial dissipation cases. Moreover, the recent works
\cite{CAOLITITI3,CAOLITITI4,CAOLITITI5} show
that the horizontal viscosity turns out
to be more crucial than the vertical one for the global well-posedness,
because the results there show that
the vertical viscosity is not required for the global
well-posedness of strong solutions to the primitive equations. Notably, the invicid primitive
equations may develop finite time singularities, see \cite{CINT,WONG}. Combining the results of \cite{CAOLITITI3,CAOLITITI4,CAOLITITI5} and those of \cite{CINT,WONG}, one can conclude that the horizontal viscosity is necessary for the global well-posedness of the primitive equations, and if ignoring the temperature effect, the horizontal viscosity is also sufficient for the global well-posedness.

\subsection{The barotropic and the first baroclinic modes interaction system}
In the tropics, the wind in the lower troposphere is of equal magnitude but with opposite sign to that in the upper troposphere, in other words, the primary effect is captured in the first baroclinic mode. However, for the study of the tropical-extratropical interactions, where the transport of momentum between the barotropic and baroclinic modes plays an important role, it is necessary to retain both the barotropic and baroclinic modes of the velocity.

Consider the primitive equations (\ref{PE}) in the layer $\mathbb R^2\times(0,H)$, for a positive constant $H$. Since we consider the tropical atmosphere and take into consideration the tropical-extratropical interactions, we can impose an ansatz of the form
\begin{equation*}
  \begin{pmatrix}
  \textbf{V} \\
  \Phi \\
\end{pmatrix}(x,y,z,t)=
\begin{pmatrix}
  u \\
  p \\
\end{pmatrix}(x,y,t)+
\begin{pmatrix}
  v \\
  p_1 \\
\end{pmatrix}(x,y,t)\sqrt2\cos(\pi z/H)
\end{equation*}
and
\begin{equation*}
  \begin{pmatrix}
  W \\
  \Theta \\
\end{pmatrix}(x,y,z,t)=
\begin{pmatrix}
  w \\
  \theta \\
\end{pmatrix}(x,y,t)\sqrt2\sin(\pi z/H),
\end{equation*}
which carry the barotropic and first baroclinic modes of the unknowns.

By performing the Galerkin projection of the primitive equations in the vertical direction onto the barotropic mode and the first baroclinic mode, one derives the following dimensionless interaction, between the barotropic mode and the first baroclinic mode, system  for the tropical atmosphere (see \cite{MAJBOOK} and also \cite{MAJBIE,FRIMAJPAU,KHOMAJ1,STEMAJ} for the details):
\begin{equation}\label{1.1}
\left\{
\begin{array}{l}
  \partial_tu+(u\cdot\nabla)u-\Delta u+\nabla p+\nabla\cdot(v\otimes v) =0,\\
  \nabla\cdot u=0,\\
  \partial_tv+(u\cdot\nabla)v-\Delta v+(v\cdot\nabla)u =\nabla\theta,\\
  \partial_t\theta+u\cdot\nabla\theta-\nabla\cdot v=S_\theta,
  \end{array}
  \right.
\end{equation}
where $u=(u_1, u_2)$ is the barotropic velocity, and $v=(v_1, v_2)$, $p$ and $\theta$, respectively, are the first baroclinic modes of the velocity, pressure and the temperature. The system is now defined on $\mathbb R^2$, and the operators $\nabla$ and $\Delta$ are therefore those for the variables $x$ and $y$.

\subsection{The moisture equation}
An important ingredient of the tropical atmospheric circulation is the water vapour. Water vapour is the most abundant greenhouse gas in the atmosphere, and it is responsible for  amplifying the long-term warming or cooling cycles. Therefore, one should also consider the coupling with an equation modeling moisture in the atmosphere.

Following \cite{FRIMAJPAU}, we couple  system (\ref{1.1}) with the following large-scale moisture equation
\begin{equation}
  \partial_tq+u\cdot\nabla q+\bar Q\nabla\cdot v=-P,\label{1.5}
\end{equation}
where $\bar Q$ is the prescribed gross moisture stratification. The precipitation rate $P$ is parameterized, according to \cite{FRIMAJPAU,STEMAJ,NEEZEN,KHOMAJ2}, as
\begin{equation}\label{EQ}
P=\frac{1}{\varepsilon}(q-\alpha\theta-\hat q)^+,
\end{equation}
where $f^+=\max\{f, 0\}$ denotes the positive part of $f$, $\varepsilon$ is a convective adjustment time scale parameter, and $\alpha$ and $\hat q$ are constants, with $\hat q>0$.

In order to close system (\ref{1.1})--(\ref{1.5}), one still needs to parameterize the source term $S_\theta$ in the temperature equation. Generally, the temperature source $S_\theta$ combines  three kinds of effects: the radiative cooling, the sensible heat flux and the precipitation $P$. For simplicity, and as in  \cite{FRIMAJPAU,MAJSOU}, we only consider in this paper the precipitation source term, i.e., we set
$$
S_\theta=P,
$$
with $P$ given by (\ref{EQ}).

As in \cite{FRIMAJPAU,MAJSOU}, by introducing the equivalent temperature $T_e$ and the equivalent moisture $q_e$ as
$$
T_e=q+\theta,\quad q_e=q-\alpha\theta-\hat q,
$$
system (\ref{1.1})--(\ref{1.5}) can be rewritten as
\begin{eqnarray}
  &&\partial_tu+(u\cdot\nabla)u-\Delta u+\nabla p+\nabla\cdot(v\otimes v)=0,\label{eq1}\\
  &&\nabla\cdot u=0,\label{eq2}\\
  &&\partial_tv+(u\cdot\nabla)v-\Delta v+(v\cdot\nabla)u=\frac{1}{1+\alpha}\nabla(T_e-q_e),\label{eq3}\\
  &&\partial_tT_e+u\cdot\nabla T_e-(1-\bar Q)\nabla\cdot v=0,\label{eq4}\\
  &&\partial_tq_e+u\cdot\nabla q_e+(\bar Q+\alpha)\nabla\cdot v=-\frac{1+\alpha}{\varepsilon} q_e^+, \label{eq5}
\end{eqnarray}
in $\mathbb R^2\times(0,\infty)$, where the constants $\alpha$ and $\bar Q$ are required to satisfy (see \cite{FRIMAJPAU})
\begin{equation}
  \label{req}
   0<\bar Q<1,\quad \alpha+\bar Q>0.
\end{equation}

\subsection{Main results}
We will work in the framework of strong solutions, which are defined below.

\begin{definition}
  \label{def}
  Given a positive time $\mathcal T$ and the initial data $(u_0, v_0, T_{e,0}, q_{e,0})$. A function $(u, v, T_e, q_e)$ is called a strong solution to system (\ref{eq1})--(\ref{eq5}), on $\mathbb R^2\times(0,\mathcal T)$, with initial data $(u_0, v_0, T_{e,0}, q_{e,0})$, if it enjoys the following regularities
  \begin{eqnarray*}
  &(u, v)\in C([0,\mathcal T]; H^1(\mathbb R^1))\cap L^2(0,T; H^2(\mathbb R^2)),\\
  &(\partial_tu,\partial_tv,\partial_tT_e, \partial_tq_e)\in L^2(0,T; L^2(\mathbb R^2)), \\
  &(T_e,q_e)\in C([0,\mathcal T]; L^2(\mathbb R^2))\cap L^\infty(0,T; H^1(\mathbb R^2)),
\end{eqnarray*}
and satisfies equations (\ref{eq1})--(\ref{eq5}), a.e.~on $\mathbb R^2\times(0,\mathcal T)$, and has the initial value
$$(u,v,T_e,q_e)|_{t=0}=(u_0, v_0, T_{e,0}, q_{e,0}).$$
\end{definition}

\begin{definition}
A function $(u, v, T_e, q_e)$ is called a global strong solution to system (\ref{eq1})--(\ref{eq5}), if it is a strong solution to system (\ref{eq1})--(\ref{eq5}), on $\mathbb R^2\times(0,\mathcal T)$, for any positive time $\mathcal T$.
\end{definition}

Throughout this paper, for positive integer $k$ and positive $q\in[1,\infty]$, we use $L^q(\mathbb R^2)$ and $W^{k,q}(\mathbb R^2)$ to denote the standard Lebesgue and Sobolev spaces, respectively, and when $q=2$, we use $H^k(\mathbb R^2)$, instead of $W^{k,2}(\mathbb R^2)$. For simplicity, we usually use $\|f\|_q$ to denote the $\|f\|_{L^q(\mathbb R^2)}$.

The first main result of this paper is on the global existence, uniqueness and well-posedness of strong solutions to the Cauchy problem of system (\ref{eq1})--(\ref{eq5}):

\begin{theorem}
  \label{glopositive}
Suppose that (\ref{req}) holds, and the initial data
\begin{equation}\label{asum}
(u_0, v_0, T_{e,0}, q_{e,0})\in H^1(\mathbb R^2),\quad\mbox{with}\quad \nabla\cdot u_0=0.
\end{equation}
Then, we have the following:

(i) There is a unique global strong solution $(u,v,T_e,q_e)$ to system (\ref{eq1})--(\ref{eq5}), with initial data $(u_0, v_0, T_{e,0}, q_{e,0})$, such that
\begin{align*}
\sup_{0\leq t\leq\mathcal T}&\|(u, v,T_e,q_e)(t)\|_{H^1}^2
+\int_0^\mathcal T\bigg(\frac{\|q_e^+\|_{H^1}^2}{\varepsilon}+
\|(u,v)\|_{H^2}^2+\|\nabla u\|_\infty\bigg)dt\\
&+\int_0^\mathcal T\|(\partial_tu,\partial_tv,\partial_t T_e)\|_2^2  dt
\leq C\left(\alpha,\bar Q, \mathcal T, \|(u_0, v_0, T_{e,0}, q_{e,0})\|_{H^1}\right),
\end{align*}
for any positive time $\mathcal T$, here and what follows, we use $C(\cdots)$ to denote a general positive constant depending only on the quantities in the parenthesis.

(ii) Suppose, in addition to (\ref{asum}), that $q_{e,0}\leq0$, a.e.~on $\mathbb R^2$, then
$$
\sup_{0\leq t\leq \mathcal T}\frac{\|q_e^+(t)\|_2^2}{\varepsilon} +\int_0^\mathcal T\|\partial_tq_e\|_2^2dt
\leq C\left(\alpha,\bar Q, \mathcal T, \|(u_0, v_0, T_{e,0}, q_{e,0})\|_{H^1}\right),
$$
for any positive time $\mathcal T$.

(iii) Suppose, in addition to (\ref{asum}), that $(\nabla T_{e,0},\nabla q_{e,0})\in L^m(\mathbb R^2)$, for some $m\in(2,\infty)$, then the following estimate holds
\begin{align*}
  \sup_{0\leq t\leq\mathcal T}\|(\nabla T_{e},\nabla q_{e})(t)\|_m^2
  \leq C\left(\alpha,\bar Q, \mathcal T, \|(u_0, v_0, T_{e,0}, q_{e,0})\|_{H^1},\|(\nabla T_{e,0},\nabla q_{e,0})\|_m\right),
\end{align*}
for any positive time $\mathcal T$, and the unique strong solution $(u,v, T_e, q_e)$ depends continuously on the initial data, on any finite interval of time.
\end{theorem}

Formally, by taking the relaxation limit, as $\varepsilon\rightarrow0^+$, system (\ref{eq1})--(\ref{eq5}) will converge to the following limiting system
\begin{eqnarray}
  &&\partial_tu+(u\cdot\nabla)u-\mu\Delta u+\nabla p+\nabla\cdot(v\otimes v)=0,\label{ineq1}\\
  &&\nabla\cdot u=0,\label{ineq2}\\
  &&\partial_tv+(u\cdot\nabla)v-\mu\Delta v+(v\cdot\nabla)u=\frac{1}{1+\alpha}\nabla(T_e-q_e),\label{ineq3}\\
  &&\partial_tT_e+u\cdot\nabla T_e-(1-\bar Q)\nabla\cdot v=0,\label{ineq4}\\
  &&\partial_tq_e+u\cdot\nabla q_e+(\bar Q+\alpha)\nabla\cdot v\leq 0,\label{ineq5}\\
  &&q_e\leq0,\label{ineq6}\\
  &&\partial_tq_e+u\cdot\nabla q_e+(\bar Q+\alpha)\nabla\cdot v= 0,\quad \mbox{a.e.~on }\{q_e<0\}. \label{ineq7}
\end{eqnarray}
Note that equation (\ref{eq5}) is now replaced by three inequalities (\ref{ineq5})--(\ref{ineq7}).

Inequality (\ref{ineq5}) comes from equation (\ref{eq5}), by noticing the negativity of the term $-\frac{1+\alpha}{\varepsilon}q_e^+$, while inequality
(\ref{ineq6}) is derived by multiplying both sides of equation (\ref{eq5}) by
$\varepsilon$, and taking the formal limit $\varepsilon\rightarrow0^+$. Inequality
(\ref{ineq7}) can be derived by the following heuristic argument: Let
$(u_\varepsilon, v_\varepsilon, T_{e\varepsilon}, q_{e\varepsilon})$ be a solution
to system (\ref{eq1})--(\ref{eq5}), and suppose that $(u_\varepsilon,
v_\varepsilon, T_{e\varepsilon}, q_{e\varepsilon})$ converges to $(u, v, T_e,
q_e)$, with $q_e\leq0$; for any compact subset $K$ of the set $\{(x,y,t)\in\mathbb R^2\times(0,\infty)~|~q_e(x,y,t)<0\}$, since
$q_{e\varepsilon}$ converges to $q_e$, one may have  $q_{e\varepsilon}<0$ on $K$,
for sufficiently small positive $\varepsilon$; therefore, by equation (\ref{eq5}),
it follows that $\partial_tq_{e\varepsilon}+u_\varepsilon\cdot\nabla
q_{e\varepsilon}+(\bar Q+\alpha)\nabla\cdot v_\varepsilon=0$, a.e.~on $K$, from
which, by taking $\varepsilon\rightarrow0^+$, one can see that (\ref{ineq7}) is
satisfied, a.e.~on $K$, and further a.e.~on $\{q_e<0\}$.

The other aim of this paper is to prove the global existence and uniqueness of strong solutions to the limiting system (\ref{ineq1})--(\ref{ineq7}), and rigorously justify the above formal convergences, as $\varepsilon\rightarrow0^+$. Strong solutions to system (\ref{ineq1})--(\ref{ineq7}) are defined in the similar way as those to system (\ref{eq1})--(\ref{eq5}).

\begin{theorem}\label{glozero}
Suppose that (\ref{req}) holds, and the initial data
\begin{equation}\label{asum1}
(u_0, v_0, T_{e,0}, q_{e,0})\in H^1(\mathbb R^2),\quad \nabla\cdot u_0=0,\quad q_{e,0}\leq0,\mbox{ a.e.~on }\mathbb R^2.
\end{equation}
Then, there is a unique global strong solution $(u,v,T_e,q_e)$ to system (\ref{ineq1})--(\ref{ineq7}), with initial data $(u_0, v_0, T_{e,0}, q_{e,0})$, such that
\begin{align*}
  \sup_{0\leq t\leq\mathcal T} \|(u, v,T_{e},q_{e})(t)\|_{H^1}^2
  & +\int_0^\mathcal T\left(
  \|(u,v)\|_{H^2}^2+\|\nabla u\|_\infty+\|(\partial_tu,\partial_tv,\partial_tT_e,
  \partial_tq_{e})\|_2^2\right) dt\\
  \leq&C\left(\alpha,\bar Q, \mathcal T, \|(u_0, v_0, T_{e,0}, q_{e,0})\|_{H^1}\right),
\end{align*}
for any positive time $\mathcal T$.

If we assume, in addition, that $(\nabla T_{e,0}, \nabla q_{e,0})\in L^m(\mathbb R^2)$, for some $m\in(2,\infty)$, then we have further that
\begin{align*}
  \sup_{0\leq t\leq\mathcal T}\|(\nabla T_{e},\nabla q_{e})(t)\|_m^2
  \leq C\left(\alpha,\bar Q, \mathcal T, \|(u_0, v_0, T_{e,0}, q_{e,0})\|_{H^1},\|(\nabla T_{e,0},\nabla q_{e,0})\|_m\right),
\end{align*}
for any positive time $\mathcal T$, and the unique strong solution $(u, v, T_e, q_e)$ depends continuously on the initial data.
\end{theorem}

\begin{theorem}
  \label{cong}
 Suppose that (\ref{req}) holds and the initial data
\begin{eqnarray*}
  &&(u_0, v_0, T_{e,0}, q_{e,0})\in H^1(\mathbb R^2),\quad \nabla\cdot u_0=0,\\
  &&(\nabla T_{e,0},\nabla q_{e,0})\in L^m(\mathbb R^2),\quad q_{e,0}\leq0,\mbox{ a.e.~on }\mathbb R^2,
\end{eqnarray*}
for some $m\in(2,\infty)$. Denote by $(u_\varepsilon, v_\varepsilon, T_{e\varepsilon}, q_{e\varepsilon})$ and $(u, v, T_e, q_e)$ the unique global strong solutions to systems (\ref{eq1})--(\ref{eq5}) and (\ref{ineq1})--(\ref{ineq7}), respectively, with the same initial data $(u_0, v_0, T_{e,0}, q_{e,0})$.

Then, we have the estimate
\begin{align*}
  &\sup_{0\leq t\leq\mathcal T}\|(u_\varepsilon-u, v_\varepsilon-v, T_{e\varepsilon}-T_e,q_{e\varepsilon}-q_e)(t)\|_2^2\\
  &+\int_0^\mathcal T\left(\|(\nabla(u_\varepsilon-u),\nabla(v_\varepsilon-v))\|_2^2 +\frac{\|q_{e\varepsilon}^+\|_2^2}{\varepsilon}\right)dt\leq C\varepsilon,
\end{align*}
for any finite positive time $\mathcal T$, where $C$ is a positive constant depending only on $\alpha, \bar Q, m, \mathcal T$, and the initial norm $\|(u_0, v_0,q_{e,0}, T_{e,0})\|_{H^1}+\|(\nabla T_{e,0},\nabla q_{e,0})\|_m$.

Therefore, in particular, we have the convergences
\begin{eqnarray*}
  &(u_\varepsilon, v_\varepsilon)\rightarrow(u, v)\quad\mbox{in }L^\infty(0,\mathcal T; L^2(\mathbb R^2))\cap L^2(0,\mathcal T; H^1(\mathbb R^2)),\\
  &(T_{e\varepsilon}, q_{e\varepsilon})\rightarrow(T_e,q_e)\quad\mbox{in }L^\infty(0,\mathcal T; L^2(\mathbb R^2)),\quad q_{e\varepsilon}^+\rightarrow0\quad\mbox{in }L^2(0,\mathcal T; L^2(\mathbb R^2)),
\end{eqnarray*}
for any positive time $\mathcal T$, and the convergence rate is of order $O(\sqrt\varepsilon)$.
\end{theorem}

\begin{remark}
(i) In the absence of the barotropic mode, global existence and uniqueness of strong solutions to the inviscid limiting system was proved in \cite{MAJSOU}, and the relaxation limit, as $\varepsilon\rightarrow0^+$, was also studied there, but the convergence rate was not achieved. Note that in the absence of the barotropic mode, the limiting system is linear, while in the presence of the barotropic mode, the limiting system is nonlinear.

(ii) Existence and uniqueness of solutions to the limiting system (\ref{ineq1})--(\ref{ineq7}), without viscosity, was proposed as an open problem in \cite{FRIMAJPAU}, and also in \cite{MAJSOU,KHOMAJSTE,MAJAB}. Notably, Theorem \ref{glozero} settles this open problem for  the \textsc{viscous version} of (\ref{ineq1})--(\ref{ineq7}). Note that we only add viscosity to the velocity equations, and we do not use any diffusivity in the temperature and moisture equations.
\end{remark}

\begin{remark}
  Global well-posedness of strong solutions to a coupled system of the primitive equations with moisture (therefore, it is a different system from those considered in this paper) was recently addressed in \cite{ZHKTZ}, where the system under consideration has full dissipation in all dynamical equations, and in particular has diffusivity in the temperature and moisture equations. Note that we do not need any diffusivity in the temperature and moisture equations in order to establish global regularity of the systems considered in this paper. It is worth mentioning that the global regularity of the coupled three-dimensional  primitive equations with moisture and with partial dissipation is a subject of a forthcoming paper.
\end{remark}

The rest of this paper is organized as follows: in section \ref{preliminary}, we state and prove several preliminary lemmas, while the proofs of Theorem \ref{glopositive}, Theorem \ref{glozero} and Theorem \ref{cong} are given in section \ref{globalpositive}, section \ref{globalzero} and section \ref{convergence}, respectively. The last section is an appendix in which  we prove some parabolic estimates that are used in this paper, and which are of general interest on their own.

\section{Preliminaries}
\label{preliminary}

%
%
%

We will frequently use the following Ladyzhenskaya inequality (see, e.g., \cite{LADYZHENSKAYA})
$$
\|f\|_{L^4(\mathbb R^2)}\leq C\|f\|_{L^2(\mathbb R^2)}^{\frac12}\|\nabla f\|_{L^2(\mathbb R^2)}^{\frac12},\quad\forall f\in H^1(\mathbb R^2).
$$

The following lemma on the Gronwall type inequality will be used to establish the global in time {\it a priori} estimates to the strong solutions to system (\ref{eq1})--(\ref{eq5}) later.

\begin{lemma}
  \label{gron}
Given a positive time $\mathcal T$, a positive integer $n$ and positive numbers $r_i\in[1,\infty), 1\leq i\leq n$. Let $a_0, a_i$
and $b_i$, $1\leq i\leq n$, be nonnegative functions,
such that $a_0, a_i\in L^\infty((0,\mathcal T))$ and $b_i\in L^1((0,\mathcal T))$.
Suppose that the
nonnegative measurable function $f$ satisfies
$$
f(t)\leq a_0(t)+\sum_{i=1}^na_i(t)\left(\int_0^tb_i(s)f^{r_i}(s)ds \right)^{\frac{1}{r_i}},
$$
for any $t\in[0,\mathcal T]$. Then, the following holds
\begin{align*}
\|f\|_{L^\infty((0,\mathcal T))}\leq& (n+1)^{r-1}\|a_0\|_\infty^r\exp\left\{ (n+1)^{r-1} \sum_{i=1}^n\|a_i\|_\infty^r (1+\|b_i\|_1)^{r+1}\right\},
\end{align*}
where $r=\max_{1\leq i\leq n}r_i$, and $\|\cdot\|_1$ and $\|\cdot\|_\infty$ denote the $L^1((0,\mathcal T))$ and $L^\infty((0,\mathcal T))$ norms, respectively.
\end{lemma}

\begin{proof}
By the H\"older and Young inequalities, we deduce
\begin{align*}
  \left(\int_0^tb_i(s)f^{r_i}(s)ds\right)^{\frac{1}{r_i}} =&\left[ \int_0^tb_i^{\frac{r-r_i}{r}}(s)\left(b_i^{\frac{1}{r}}(s)f(s)\right) ^{r_i}ds\right]^{\frac{1}{r_i}}\\
  \leq&\left(\int_0^tb_i(s)ds\right)^{\frac{r-r_i}{rr_i}}
  \left(\int_0^tb_i(s)f^{r}(s)ds\right)^{\frac{1}{r}}\\
  \leq&\left(1+\|b_i\|_1\right)\left(\int_0^tb_i(s)f^{r}(s)ds\right)^ {\frac 1r},
\end{align*}
for $1\leq i\leq n$. Therefore, by assumption, we have
\begin{align*}
  f(t)\leq&\|a_0\|_\infty+\sum_{i=1}^n\|a_i\|_\infty (1+\|b_i\|_1)\left(\int_0^tb_i(s) f^{r}(s)ds\right)^{\frac 1r},
\end{align*}
from which, taking the $r$-th powers to both sides of the above inequality, and using the elementary inequality $\left(\sum_{i=0}^nc_i\right)^r\leq(n+1)^{r-1}\sum_{i=0}^nc_i^r$, where $c_i$ are positive numbers,
we arrive at
$$
f^r(t)\leq (n+1)^{r-1}\|a_0\|_\infty^r+(n+1)^{r-1} \sum_{i=1}^n\|a_i\|_\infty^r (1+\|b_i\|_1)^r\left(\int_0^tb_i(s) f^{r}(s)ds\right).
$$
Applying the Gronwall inequality to the above inequality, we have
\begin{align*}
  f^r(t)\leq&(n+1)^{r-1}\|a_0\|_\infty^r\exp\left\{ (n+1)^{r-1} \sum_{i=1}^n\|a_i\|_\infty^r (1+\|b_i\|_1)^r\int_0^tb_i(s)ds\right\}\\
  \leq&(n+1)^{r-1}\|a_0\|_\infty^r\exp\left\{ (n+1)^{r-1} \sum_{i=1}^n\|a_i\|_\infty^r (1+\|b_i\|_1)^{r+1}\right\},
\end{align*}
from which, taking the $r$-th power root to both sides of the above inequality, and taking the supremum with respective to $t$ over $(0,\mathcal T)$, one obtains the conclusion.
\end{proof}

The next lemma will be employed to prove the uniqueness of strong solutions.

\begin{lemma}\label{lemuniq}
Given a positive time $\mathcal T$, and let $m_1, m_2$ and $S$ be nonnegative functions on $(0,\mathcal T)$, such that
$$
m_1, S\in L^1((0,\mathcal T)),\quad m_2\in L^2((0,\mathcal T)), \mbox{ and } S>0, \mbox{ a.e.~on } (0,\mathcal T).
$$
Suppose that $f$ and $G$ are two nonnegative functions on $(0,\mathcal T)$, with $f$ being absolutely continuous on $[0,\mathcal T)$, and satisfy
\begin{equation*}
\left\{
\begin{array}{l}
f'(t)+G(t)\leq m_1(t)f(t)+m_2(t)\left[f(t)G(t)\log^+\left(\frac{S(t)}{G(t)}\right) \right]^{\frac12},\quad\mbox{ a.e.~on }(0,\mathcal T),\\
f(0)=0,
\end{array}
\right.
\end{equation*}
where $\log^+z=\max\{0,\log z\}$, for $z\in(0,\infty)$, and when $G(t)=0$, at some time $t\in[0,\mathcal T)$, we adopt the following natural convention
$$
G(t)\log^+\left(\frac{S(t)}{G(t)}\right)=\lim_{z\rightarrow 0^+}z\log^+\left(\frac{S(t)}{z}\right)=0.
$$
Then, we have $f\equiv0$ on $[0,\mathcal T)$.
\end{lemma}

\begin{proof}
Suppose, by contradiction, that there is some time $t_*\in(0,\mathcal T)$, such that $f(t_*)>0$. Recalling that $f$ is absolutely continuous on $[0,\mathcal T)$, by the property of continuous functions, there must be a time $t_0\in[0, t_*)$, such that $f(t_0)=0$ and $f(t)>0$, for any $t\in(t_0,t_*]$. In the rest of the proof, we will focus on the time interval $[t_0, t_*)$.
For any $\sigma\in(0,\infty)$, one can easily check that
$$
\log^+z\leq\frac{z^\sigma}{\sigma e}, \quad\mbox{for }z\in(0,\infty).
$$
Recall the Young inequality of the form $ab\leq\frac{a^p}{p}+\frac{b^q}{q}$, for any nonnegative numbers $a,b$, and for any $p,q\in(1,\infty)$, with $\frac1p+\frac1q=1$. Thanks to the above inequality, and choosing $\sigma\in(0,1)$, it follows from the assumption and the Young inequality that
\begin{eqnarray*}
  f'+G&\leq&m_1f+m_2\left[fG\frac{1}{e\sigma}\left(\frac{S}{G} \right)^\sigma\right]^{\frac12}=m_1f+m_2S^{\frac\sigma2}G^{ \frac{1-\sigma}{2}}\left(\frac{f}{e\sigma}\right)^{\frac12}\\
  &\leq&m_1f+\frac{1-\sigma}{2}G+\frac{1+\sigma}{2}\left[m_2S^{\frac\sigma2} \left(\frac{f}{e\sigma}\right)^{\frac12}\right]^{\frac{2}{1+\sigma}}\\
  &=&m_1f+\frac{1-\sigma}{2}G+\frac{1+\sigma}{2} m_2^{\frac{2}{1+\sigma}}S^{\frac{\sigma} {1+\sigma}}\left(\frac{f}{e\sigma}\right)^{\frac{1}{1+\sigma}}\\
  &\leq&m_1f+G+m_2^{\frac{2}{1+\sigma}}S^{\frac{\sigma} {1+\sigma}}\left(\frac{f}{\sigma}\right)^{\frac{1}{1+\sigma}},\quad\mbox{a.e.~on }(0,\mathcal T).
\end{eqnarray*}
Note that the arguments used in the above inequality are for the time when $G(t)>0$; however, for the time when $G(t)=0$, recalling that we understood the term involving $G$ as zero, therefore, the above inequality result holds trivially. Therefore, we obtain
$$
f'\leq m_1f+m_2^{\frac{2}{1+\sigma}}S^{\frac{\sigma} {1+\sigma}}\left(\frac{f}{\sigma}\right)^{\frac{1}{1+\sigma}},
$$
for any $\sigma\in(0,1)$, and for a.e.~$t\in[t_0,t_*)$. Recall that $f(t)>0$, for $t\in(t_0, t_*)$. Dividing both sides of the above inequality by $f^{\frac{1}{1+\sigma}}$, then one can deduce
\begin{eqnarray*}
\left(f^{\frac{\sigma}{1+\sigma}}\right)'&\leq&\frac{\sigma}{1+\sigma} m_1f^{\frac{\sigma}{1+\sigma}}+\frac{\sigma^{\frac{\sigma}{1+\sigma}}} {1+\sigma}m_2^{\frac{2}{1+\sigma}}S^{\frac{\sigma}{1+\sigma}}\\
&\leq&\frac{\sigma}{1+\sigma} m_1f^{\frac{\sigma}{1+\sigma}}+ \sigma^{\frac{\sigma}{1+\sigma}} m_2^{\frac{2}{1+\sigma}}S^{\frac{\sigma}{1+\sigma}},
\end{eqnarray*}
for a.e.~$t\in(t_0,t_*)$. Applying the Gronwall inequality to the above inequality, and recalling that $f(t_0)=0$, it follows from the H\"older inequality that
\begin{eqnarray*}
f^{\frac{\sigma}{1+\sigma}}(t)&\leq& \sigma^{\frac{\sigma}{1+\sigma}}e^{\frac{\sigma}{1+\sigma} \int_{t_0}^tm_1(s)ds} \int_{t_0}^tm_2^{\frac{2}{1+\sigma}}(s)S^{\frac{\sigma}{1+\sigma}}(s)ds\\
&\leq&\sigma^{\frac{\sigma}{1+\sigma}}e^{\frac{\sigma}{1+\sigma} \int_{t_0}^tm_1(s)ds} \left(\int_{t_0}^tm_2^2(s)ds\right)^{\frac{1}{1+\sigma}} \left(\int_0^tS(s)ds\right)^{\frac{\sigma}{1+\sigma}}
\end{eqnarray*}
from which, taking the $\frac{1+\sigma}{\sigma}$-th power to both sides of the above inequality, one obtains
\begin{eqnarray*}
  &&f(t)\leq \sigma e^{\int_{t_0}^tm_1(s)ds} \left(\int_{t_0}^tm_2^2(s)ds\right)^{\frac{1}{\sigma}} \int_0^tS(s)ds ,
\end{eqnarray*}
for any $t\in[t_0, t_*)$, and for any $\sigma\in(0,1)$. Recall that $m_2\in L^2((0,\mathcal T))$, by the absolute continuity of the integrals, there is a positive number $\eta\leq t_*-t_0$, such that $\int_{t_0}^tm_2^2(s)ds\leq1$, for any $t\in[t_0,t_0+\eta)$. Therefore, the above inequality implies
$$
f(t)\leq \sigma e^{\int_{t_0}^tm_1(s)ds}\int_0^tS(s)ds,
$$
for any $t\in[t_0,t_0+\eta)$, and for any $\sigma\in(0,1)$. By taking $\sigma\rightarrow0^+$, this implies that $f\equiv0$, for any $t\in[t_0,t_0+\eta)$, which contradicts the assumption that $f(t)>0$, for any $t\in(t_0, t_*)$. This contradiction implies that there is no such $t_*\in(0,\mathcal T)$ that $f(t_*)>0$, in other words, recalling that $f$ is a nonnegative function, we have $f\equiv0$ on $[0,\mathcal T)$. This completes the proof.
\end{proof}

We also will use the following elementary lemma.

\begin{lemma}
  \label{zeroa.e.}
Let $\Omega\subseteq\mathbb R^d$ be a measurable set of positive measure, and $f$ be a measurable function defined on $\Omega$. Suppose that, for any positive number $\eta$, there is a measurable subset $E_\eta$ of $\Omega$, with $|E_\eta|\leq\eta$, such that $f=0$, a.e.~on $\Omega\setminus E_\eta$. Then, $f=0$, a.e.~on $\Omega$.
\end{lemma}

\begin{proof}
  Suppose, by contradiction, that the conclusion does not hold. Then there is a subset $E$ of $\Omega$, with $0<|E|<\infty$, such that $|f|>0$ on $E$, here $|E|$ denotes the $L^d$-Lebessgue measure of the subset $E$. Then, for $\eta=\frac{|E|}{2}$, by assumption, there is a subset $E_\eta$ of $\Omega$, with $|E_\eta|\leq\eta$, such that $f=0$ on $\Omega\setminus E_\eta$. This implies that $E\subseteq E_\eta$, and thus
  $$
  |E|\leq|E_\eta|\leq \eta=\frac{|E|}{2}.
  $$
  Therefore, $|E|=0$, which contradicts the assumption that $|E|>0$. This contradiction implies the conclusion of the lemma.
\end{proof}

\section{Global existence and uniqueness of the system with positive $\varepsilon$}
\label{globalpositive}
In this section, we will prove the global existence and uniqueness of strong solutions to the Cauchy problem of
system (\ref{eq1})--(\ref{eq5}), for any positive $\varepsilon$. Several
$\varepsilon$-independent
{\it a priori} estimates will also be obtained.

Let's start with the following result on the local existence and uniqueness of strong solutions to the Cauchy problem to system (\ref{eq1})--(\ref{eq5}).

\begin{proposition}\label{loc}
Suppose that (\ref{req}) holds. Then, for any initial data
\begin{eqnarray*}
  &&(u_0,v_0,T_{e,0},q_{e,0})\in H^1(\mathbb R^2),\quad \mbox{with }\nabla\cdot u_0=0,
\end{eqnarray*}
there is a unique local strong solution $(u,v,T_e,q_e)$ to system (\ref{eq1})--(\ref{eq5}), on $\mathbb R^2\times(0,\mathcal T)$, with initial data $(u_0,v_0,T_{e,0},q_{e,0})$,
where the existence time $\mathcal T$
depends on $\alpha$, $Q$, $\varepsilon$ and the initial norm $\|(u_0, v_0, T_{e,0}, q_{e,0})\|_{H^1}$.
\end{proposition}

\begin{proof}
\textbf{(i) The existence. }The existence of strong solutions to system (\ref{eq1})--(\ref{eq5}), with initial data $(u_0, v_0, T_{e,0}, q_{e,0})$ can be proven by the standard regularization argument as follows: (i) adding the diffusivity terms $-\eta\Delta T_e$ and $-\eta \Delta q_e$ to the left-hand sides of equations (\ref{eq4}) and (\ref{eq5}), respectively, in other words, we consider the following regularized system
\begin{equation}\label{rs}
\left\{
\begin{array}{l}
  \partial_tu+(u\cdot\nabla)u-\Delta u+\nabla p+\nabla\cdot(v\otimes v)=0,\\
  \nabla\cdot u=0,\\
  \partial_tv+(u\cdot\nabla)v-\Delta v+(v\cdot\nabla)u=\frac{1}{1+\alpha}\nabla(T_e-q_e),\\
  \partial_tT_e+u\cdot\nabla T_e-(1-\bar Q)\nabla\cdot v-\eta\Delta T_e=0,\\
  \partial_tq_e+u\cdot\nabla q_e+(\bar Q+\alpha)\nabla\cdot v-\eta \Delta q_e =-\frac{1+\alpha}{\varepsilon} q_e^+;
  \end{array}
  \right.
\end{equation}
(ii) for each $\eta>0$, the Cauchy problem of the regularized system (\ref{rs}), with initial data $(u_0, v_0, T_{e,0}, q_{e,0})$, has a unique short time strong solution $(u^{(\eta)}, v^{(\eta)}, T_e^{(\eta)}, q_e^{(\eta)})$, which satisfies some $\eta$-independent {\it a priori} estimates, on some $\eta$-independent time interval $(0,\mathcal T)$, for a positive time $\mathcal T$ depending only on on $\alpha$, $Q$, $\varepsilon$ and the initial norm $\|(u_0, v_0, T_{e,0}, q_{e,0})\|_{H^1}$; (iii) thanks to these $\eta$-independent estimates, by adopting the Cantor diagonal argument,
one can apply the Aubin-Lions lemma and take the limit $\eta\rightarrow0^+$ to show the local existence of strong solutions to the Cauchy problem of system (\ref{eq1})--(\ref{eq5}), with initial data $(u_0,v_0,T_{e,0},q_{e,0})$. Since the proof is standard, we omit it here; however, the key part of the proof, i.e., the relevant {\it a priori} estimates, are essentially contained in the ''formal" proofs of Propositions \ref{basicest}--\ref{h1v}, below. As it was mentioned above, these formal estimates can be rigorously justified by establishing them first, to be $\eta-$independent,  for the regularized system (\ref{rs}) and then passing with the limit as $\eta \to 0^+$.

\textbf{(ii) The uniqueness. }Let $(u, v, T_e, q_e)$ and $(\tilde u, \tilde v, \tilde T_e, \tilde q_e)$ be two strong solutions to system (\ref{eq1})--(\ref{eq5}), with the same initial data $(u_0, v_0, T_{e,0}, q_{e, 0})$, on the time interval $(0,\mathcal T)$. Define the new functions
$$
(\delta u, \delta v, \delta T_e,\delta q_e)=(u, v, T_e, q_e)-(\tilde u, \tilde v, \tilde T_e, \tilde q_e).
$$
Then, one can easily check that
\begin{eqnarray}
&\partial_t\delta u+(u\cdot\nabla)\delta u +(\delta u\cdot\nabla)\tilde u-\Delta\delta u+\nabla\delta p+\nabla\cdot(v\otimes\delta v+\delta v\otimes\tilde v)=0,\label{deq1}\\
&\nabla\cdot\delta u =0,\label{deq2}\\
&\partial_t\delta v+(u\cdot\nabla)\delta v+(\delta u\cdot\nabla)\tilde v-\Delta\delta v+(v\cdot\nabla)\delta u\nonumber\\
&+(\delta v\cdot\nabla)\tilde u=\frac{1}{1+\alpha}\nabla(\delta T_e-\delta q_e), \label{deq3}\\
&\partial_t\delta T_e+u\cdot\nabla\delta T_e +\delta u\cdot\nabla\tilde T_e-(1-\bar Q)\nabla\cdot\delta v=0,\label{deq4}\\
&\partial_t\delta q_e+u\cdot\nabla\delta q_e +\delta u\cdot\nabla\tilde q_e+(\bar Q+\alpha)\nabla\cdot\delta v=-\frac{1+\alpha}{\varepsilon}(q_e^+-\tilde q_e^+).\label{deq5'}
\end{eqnarray}

Since equations (\ref{deq1})--(\ref{deq4}) hold in $L^2(0,\mathcal T; L^2(\mathbb R^2))$, we multiply equations (\ref{deq1}), (\ref{deq3}) and (\ref{deq4}) by $\delta u$, $\delta v$ and $\delta T_e$, respectively, and integrating over $\mathbb R^2$, then it follows from integration by parts that
\begin{align*}
  &\frac12\frac{d}{dt}(\|\delta u\|_2^2+\|\delta v\|_2^2+\|\delta T_e\|_2^2)+\|\nabla\delta u\|_2^2+\|\nabla\delta v\|_2^2\\
  =&-\int_{\mathbb R^2}[(\delta u\cdot\nabla)\tilde u+\nabla\cdot(v\otimes\delta v+\delta v\otimes\tilde v)]\cdot\delta udxdy\\
  &-\int_{\mathbb R^2}\left\{[(\delta u\cdot\nabla)\tilde v+(v\cdot\nabla)\delta u+(\delta u\cdot\nabla)\tilde u]\cdot\delta v+\frac{\delta T_e-\delta q_e}{1+\alpha}\nabla\cdot\delta v\right\} dxdy\\
  &-\int_{\mathbb R^2}[\delta u\cdot\nabla\tilde T_e\delta T_e-(1-\bar Q) \nabla\cdot\delta v]\delta T_edxdy=:I.
  \end{align*}
  By the Young inequality, we deduce
  \begin{align*}
  I\leq&\int_{\mathbb R^2}[|\delta u||\nabla\tilde u|+(|v|+|\tilde v|)|\nabla\delta v|+(|\nabla v|+|\nabla\tilde v|)|\delta v|]|\delta u|dxdy\\
  &+\int_{\mathbb R^2}\left\{[|\delta u|(|\nabla\tilde v|+|\nabla\tilde u|)+|v||\nabla\delta u|]|\delta v|+\frac{|\nabla\delta v|}{1+\alpha}(|\delta T_e|+|\delta q_e|) \right\}dxdy\\
  &+\int_{\mathbb R^2}[|\delta u||\nabla\tilde T_e||\delta T_e|+(1-\bar Q)|\nabla\delta v||\delta T_e|]dxdy\\
  \leq&\frac12\int_{\mathbb R^2}(|\nabla\delta u|^2+|\nabla\delta v|^2)dxdy+C\int_{\mathbb R^2}[(|\nabla\tilde u|+|\nabla\tilde v|+|\nabla v|+|v|^2\\
  &+|\tilde v|^2)(|\delta u|^2+|\delta v|^2)+|\delta T_e|^2+|\delta q_e|^2+|\nabla\tilde T_e||\delta u||\delta T_e|]dxdy,
\end{align*}
and thus
\begin{align}
  &\frac{d}{dt}\|(\delta u,\delta v,\delta T_e)\|_2^2+\|\nabla\delta u\|_2^2+\|\nabla\delta v\|_2^2\nonumber\\
  \leq&C\int_{\mathbb R^2}[(|\nabla\tilde u|+|\nabla\tilde v|+|\nabla v|+|v|^2+|\tilde v|^2)(|\delta u|^2+|\delta v|^2)\nonumber\\
  &+|\delta T_e|^2+|\delta q_e|^2+|\nabla\tilde T_e||\delta u||\delta T_e|]dxdy.\label{uni1}
\end{align}

Multiplying equation (\ref{deq5'}) by $\delta q_e$, integrating the resultant over $\mathbb R^2$, then it follows from integration by parts and the Young inequality that
\begin{align*}
  &\frac12\frac{d}{dt}\|\delta q_e\|_2^2+\frac{1+\alpha}{\varepsilon}\int_{\mathbb R^2}(q_e^+-\tilde q_e^+)(q_e-\tilde q_e)dxdy\\
  =&-\int_{\mathbb R^2}(\delta u\cdot\nabla\tilde q_e+(\alpha+\bar Q)\nabla\cdot\delta v)\delta q_e dxdy\\
  \leq&\frac14\|\nabla\delta v\|_2^2+C\int_{\mathbb R^2}(|\delta q_e|^2+|\nabla\tilde q_e||\delta u||\delta q_e|)dxdy,
\end{align*}
from which, noticing that the function $z^+$ is nondecreasing in $z$, thus $(q_e^+-\tilde q_e^+)(q_e-\tilde q_e)\geq0$, and one obtains
$$
\frac{d}{dt}\|\delta q_e\|_2^2\leq \frac12\|\nabla\delta v\|_2^2+C\int_{\mathbb R^2}(|\delta q_e|^2+|\nabla\tilde q_e||\delta u||\delta q_e|)dxdy.
$$

Summing the above inequality with (\ref{uni1}) yields
\begin{eqnarray}
  &&\frac{d}{dt}\|(\delta u,\delta v,\delta T_e,\delta q_e)\|_2^2+\frac12(\|\nabla\delta u\|_2^2+\|\nabla\delta v\|_2^2)\nonumber\\
  &\leq&C\int_{\mathbb R^2}[(|\nabla\tilde u|+|\nabla\tilde v|+|\nabla v|+|v|^2+|\tilde v|^2)(|\delta u|^2+|\delta v|^2)\nonumber\\
  &&+|\delta T_e|^2+|\delta q_e|^2+|\nabla\tilde T_e||\delta u||\delta T_e|+|\nabla\tilde q_e||\delta u||\delta q_e|]dxdy,\label{add}
\end{eqnarray}
from which, by the H\"older, Ladyzhenskay and Young inequalities, we deduce
\begin{eqnarray*}
  &&\frac{d}{dt}\|(\delta u,\delta v,\delta T_e,\delta q_e)\|_2^2+\frac12(\|\nabla\delta u\|_2^2+\|\nabla\delta v\|_2^2)\nonumber\\
  &\leq&C(\|(\nabla\tilde u, \nabla\tilde v,\nabla v)\|_2+\|(v,\tilde v)\|_4^2)\|(\delta u,\delta v)\|_4^2\nonumber\\
  &&+C\|(\delta T_e, \delta q_e)\|_2^2+C\|(\nabla\tilde T_e,\nabla\tilde q_e)\|_2\|\delta u\|_\infty\|(\delta T_e, \delta q_e)\|_2\nonumber\\
  &\leq&C(\|(\nabla\tilde u, \nabla\tilde v,\nabla v)\|_2+\|(v,\tilde v)\|_4^2)\|(\delta u,\delta v)\|_2  \|(\nabla\delta u,\nabla\delta v)\|_2\nonumber\\
  &&+C\|(\delta T_e, \delta q_e)\|_2^2+C\|(\nabla\tilde T_e,\nabla\tilde q_e)\|_2\|\delta u\|_\infty\|(\delta T_e, \delta q_e)\|_2\nonumber\\
  &\leq&\frac14\|(\nabla\delta u,\nabla\delta v)\|_2^2 +C\left(\|(\nabla\tilde u, \nabla\tilde v,\nabla v)\|_2^2+\|(\tilde u, \tilde v)\|_4^4\right)\|(\delta u,\delta v)\|_2^2\nonumber\\
  &&+C\|(\delta T_e, \delta q_e)\|_2^2+C\|(\nabla\tilde T_e,\nabla\tilde q_e)\|_2\|(\delta u,\delta v)\|_\infty\|(\delta T_e, \delta q_e)\|_2.
\end{eqnarray*}
Therefore, one has
\begin{eqnarray}
&&\frac{d}{dt}\|(\delta u,\delta v,\delta T_e,\delta q_e)\|_2^2+\frac14\|(\delta u,\delta v)\|_{H^1}^2\nonumber\\
&\leq& C\left(1+\|(\tilde u, \tilde v)\|_4^4+\|(\nabla\tilde u, \nabla\tilde v,\nabla v)\|_2^2\right)\|(\delta u,\delta v,\delta T_e, \delta q_e)\|_2^2\nonumber\\
&&+C\|(\nabla\tilde T_e,\nabla\tilde q_e)\|_2\|(\delta u,\delta v)\|_\infty\|(\delta T_e, \delta q_e)\|_2.\label{LINEW6.2}
\end{eqnarray}

Recalling the following Brezis--Gallouet--Wainger inequality (see \cite{Brezis_Gallouet_1980,Brezis_Wainger_1980})
$$
\|f\|_{L^\infty(\mathbb R^2)}\leq C\|f\|_{H^1(\mathbb R^2)}\log^{\frac12}\left(\frac{\|f\|_{H^2(\mathbb R^2)}}{\|f\|_{H^1(\mathbb R^2)}}+e\right),
$$
and denoting $U=(u,v), \tilde U=(\tilde u, \tilde v)$ and $\delta U=(\delta u, \delta v)$, we have
\begin{eqnarray}
\|\delta U\|_\infty&\leq& C\|\delta U\|_{H^1}\log^{\frac12}\left(\frac{\|\delta U\|_{H^2}}{\|\delta U\|_{H^1}}+e\right)
\leq C\|\delta U\|_{H^1}\log^{\frac12}\left(\frac{S(t)}{\|\delta U\|_{H^1}}\right)\nonumber\\
&=&C\left[\|\delta U\|_{H^1}^2\log^+\left(\frac{S(t)}{\|\delta U\|_{H^1}}\right)\right]^{\frac12}, \label{LINEW6.2-1}
\end{eqnarray}
where
$$
S(t)=\|U\|_{H^2}+\|\tilde U\|_{H^2}+e(\|U\|_{H^1}+\|\tilde U\|_{H^1}).
$$
Note that, when $\delta U\equiv 0$, (\ref{LINEW6.2-1}) still holds, as long as we understand the quantity on the right-hand side as zero, in the natural way as in Lemma \ref{lemuniq}.

Denoting
\begin{eqnarray*}
  &&f=\|(\delta u,\delta v,\delta T_e,\delta q_e)\|_2^2,\quad G=\frac14\|(\delta u,\delta v)\|_{H^1}^2,\\
  &&m_1=C\left(1+\|(\tilde u, \tilde v)\|_4^4+\|(\nabla\tilde u, \nabla\tilde v,\nabla v)\|_2^2\right),\quad m_2=C\|(\nabla\tilde T_e,\nabla\tilde q_e)\|_2,
\end{eqnarray*}
then it follows from (\ref{LINEW6.2}) and (\ref{LINEW6.2-1}) that
$$
f'+G\leq m_1f+m_2\left[fG\log^+\left(\frac{S/4}{G}\right)\right]^{\frac12}.
$$
Here, at the time when $G(t)=0$, the term involving $G(t)$ on the right-hand side of the above inequality is understood as zero, as it was in Lemma \ref{lemuniq}.
Recalling the regularities of $(u,v,T_e,q_e)$ and $(\tilde u, \tilde v, \tilde T_e, \tilde q_e)$, one can easily check, thanks to the Ladyzhanskaya inequality,  that $m_1,S\in L^1((0,\mathcal T))$ and $m_2\in L^2((0,\mathcal T))$. Therefore, we can apply Lemma \ref{lemuniq} to conclude that $f\equiv0$, which proves the uniqueness.
\end{proof}

For the rest of this section, we always suppose that
$(u,v,T_e,q_e)$ is the unique strong solution to system (\ref{eq1})--(\ref{eq5}), on $\mathbb R^2\times(0,\mathcal T)$, for some positive time $\mathcal T$, with initial data $(u_0, v_0, T_{e,0},q_{e,0})$. We are going to establish several $\varepsilon$-independent {\it a priori} estimates on $(u, v, T_e, q_e)$.
Before performing these {\it a priori} estimates, we point out, again, that the
arguments being used in the proofs of Propositions
\ref{l4est}--\ref{h1v}, below, are somewhat formal, because
$(u,v,T_e,q_e)$ may not have the required smoothness for justifying
the arguments. However, one can follow the same arguments presented in the proofs of  Propositions
\ref{l4est}--\ref{h1v} to
establish the same {\it a priori} estimates to the regularized system
(\ref{rs}), for which the solutions fulfill the required smoothness,
and then take the limit $\eta\rightarrow0^+$, recalling the weakly
lower semi-continuity of the relevant norms, to obtain the desired {\it a priori} estimates on $(u,v,T_e,q_e)$.

Let's start with the basic energy equality stated in the following proposition. We observe that here we have energy equality, instead of inequality, as in the case of strong solutions of the Navier-Stokes equations. Observe, however, that for the rest of the proof of the main result it is sufficient to have energy inequality. 

\begin{proposition}
  \label{basicest}
  We have the following estimate
  \begin{align*}
\frac12\frac{d}{dt}\bigg(\|u\|_2^2+\|v\|_2^2&+\frac{\|T_e\|_2^2}{(1+\alpha) (1-\bar Q)}+\frac{\|q_e\|_2^2}{(1+\alpha)(\bar Q+\alpha)}\bigg)\\
    &+\|\nabla u\|_2^2+\|\nabla v\|_2^2+\frac{\|q_e^+\|_2^2}{\varepsilon(\bar Q+\alpha)}=0,
  \end{align*}
  for any $t\in(0,\mathcal T)$.
\end{proposition}

\begin{proof}
  Multiplying equations (\ref{eq1}) and (\ref{eq3}) by $u$ and $v$, respectively, summing the resultants up and integrating over $\mathbb R^2$, then it follows from integration by parts that
  \begin{equation}
    \frac12\frac{d}{dt}(\|u\|_2^2+\|v\|_2^2)+\|\nabla u\|_2^2+\|\nabla v\|_2^2
    =\frac{1}{1+\alpha}\int_{\mathbb R^2}(q_e-T_e)\nabla\cdot vdxdy,\label{be1}
  \end{equation}
  where we have used the following fact
  $$
  \int_{\mathbb R^2}[\nabla\cdot(v\otimes v)\cdot u+(v\cdot\nabla )u\cdot v]dxdy
  =\int_{\mathbb R^2}[(v\cdot\nabla)u\cdot v-(v\otimes v):\nabla u]dxdy=0.
  $$
  Multiplying equation (\ref{eq4}) by $(1+\alpha)^{-1}(1-\bar Q)^{-1}T_e$, and integrating over $\mathbb R^2$, then it follows from integration by parts that
  \begin{equation}
    \label{be2}
    \frac{1}{2(1+\alpha)(1-\bar Q)}\frac{d}{dt}\|T_e\|_2^2-\frac{1}{1+\alpha}\int_{\mathbb R^2}T_e \nabla\cdot vdxdy=0.
  \end{equation}
  Multiplying equation (\ref{eq5}) by $(1+\alpha)^{-1}(\bar Q+\alpha)^{-1}q_e$, and integrating over $\mathbb R^2$, then it follows from integration by parts that
  \begin{equation}
    \label{be3}
    \frac{1}{2(1+\alpha)(\bar Q+\alpha)}\frac{d}{dt}\|q_e\|_2^2+\frac{1}{1+\alpha} \int_{\mathbb R^2}q_e\nabla\cdot v dxdy=-\frac{1}{\varepsilon(\bar Q+\alpha)}\int_{\mathbb R^2}|q_e^+|^2dxdy.
  \end{equation}
  Summing (\ref{be1})--(\ref{be3}) up yields the conclusion.
\end{proof}

As an intermediate step to obtain the $L^\infty(0,\mathcal T; H^1(\mathbb R^2))$ estimate for $(u,v,T_e,q_e)$, we prove the $L^\infty(0,\mathcal T; L^4(\mathbb R^2))$ estimate in the next proposition.

\begin{proposition}
  \label{l4est}
  Denote $U=(u,v)$. Then, we have the estimate
  \begin{align*}
    \sup_{0\leq t<\mathcal T}\|(U,T_e,q_e)\|_4^4+\int_0^\mathcal T\left(\big\||U|\nabla U\big\|_2^2+\|\nabla v\|_4^4\right)dt\leq C,
  \end{align*}
  for a positive constant $C$ depending only on the parameters $\alpha, \bar Q,\mathcal T$ and the initial norm $\|(U_0,T_{e,0},q_{e,0})\|_{ L^2(\mathbb R^2)\cap L^4(\mathbb R^2)}$, and in particular, $C$ is independent of $\varepsilon$.
\end{proposition}

\begin{proof}
  Multiplying equations (\ref{eq1}) and (\ref{eq3}) by $|U|^2u$ and $|U|^2v$, respectively, summing the resultants up and integrating over $\mathbb R^2$, then it follows from integration by parts and the H\"older inequality
  that
  \begin{align}
    &\frac14\frac{d}{dt}\|U\|_4^4+\int_{\mathbb R^2}\left(|U|^2|\nabla U|^2+\frac12|\nabla(|U|^2)|^2\right)dxdy\nonumber\\
    =&\int_{\mathbb R^2}[p\nabla\cdot(|U|^2u)-(\nabla\cdot(v\otimes v))\cdot|U|^2 u\nonumber\\
    &\qquad-(v\cdot\nabla)u\cdot|U|^2v+\frac{q_e-T_e}{1+\alpha}\nabla\cdot(|U|^2v)]dxdy \nonumber\\
    \leq&3\int_{\mathbb R^2}[|p||U|^2|\nabla U|+|U|^4|\nabla U|+(|q_e|+|T_e|)|U|^2|\nabla U|]dxdy\nonumber\\
    \leq&3\big(\|p\|_4+\big\||U|^2\big\|_4+\|T_e\|_4+\|q_e\|_4 \big)\|U\|_4\big\||U|\nabla U\big\|_2. \label{l41}
  \end{align}
  Applying the divergence operator to equation (\ref{eq1}), in view of (\ref{eq2}), one can see that
  $$
  -\Delta p=\nabla\cdot\nabla\cdot(u\otimes u+v\otimes v).
  $$
  Note that $p$ is uniquely determined by the above elliptic equation by assuming that $p\rightarrow0$, as $(x,y)\rightarrow\infty$. Thus, by the elliptic estimates, one has
  $$
  \|p\|_4\leq C\|u\otimes u+v\otimes v\|_4\leq C\big\||U|^2\big\|_4.
  $$
  Substituting this estimate into (\ref{l41}), and using the Ladyzhenskaya and Young inequalities, one deduces
  \begin{align*}
    &\frac14\frac{d}{dt}\|U\|_4^4+\big\||U|\nabla U\big\|_2^2+\frac12\big\|\nabla|U|^2\big\|_2^2\\
    \leq&C\big(\big\||U|^2\big\|_4+\|T_e\|_4+\|q_e\|_4 \big)\|U\|_4\big\||U|\nabla U\big\|_2\\
    \leq&C\big(\big\||U|^2\big\|_2^{\frac12}\big\|\nabla|U|^2\big\|_2^{\frac12} +\|T_e\|_4+\|q_e\|_4 \big)\|U\|_4\big\||U|\nabla U\big\|_2\\
    \leq&\frac12\big(\big\||U|\nabla U\big\|_2^2+\big\|\nabla|U|^2\big\|_2^2\big)+C[\|U\|_4^2 (\|T_e\|_2^2+\|q_e\|_2^2)+\|U\|_4^8]\\
    \leq&C(1+\|U\|_4^4) (\|T_e\|_2^2+\|q_e\|_2^2+\|U\|_4^4)\\
    &+\frac12\big(\big\||U|\nabla U\big\|_2^2+\big\|\nabla|U|^2\big\|_2^2\big),
  \end{align*}
  and thus
  \begin{align}
    \frac{d}{dt}\|U\|_4^4+2\big\||U|\nabla U\big\|_2^2\leq C(1+\|U\|_4^4) (\|T_e\|_2^2+\|q_e\|_2^2+\|U\|_4^4). \label{l42}
  \end{align}

  Multiplying equation (\ref{eq4}) by $|T_e|^2T_e$, and integrating over $\mathbb R^2$, then it follows from integration by parts and the H\"older inequality that
  \begin{equation*}
    \frac14\frac{d}{dt}\|T_e\|_4^4=(1-\bar Q)\int_{\mathbb R^2}\nabla\cdot v|T_e|^2T_edxdy\leq(1-\bar Q)\|\nabla v\|_4\|T_e\|_4^3,
  \end{equation*}
  which implies
  \begin{equation}
    \label{l43}
    \frac{d}{dt}\|T_e\|_4^2\leq2(1-\bar Q)\|\nabla v\|_4\|T_e\|_4.
  \end{equation}
  Similar manipulation to equation (\ref{eq5}) yields
  \begin{equation}
    \label{l44}
    \frac{d}{dt}\|q_e\|_4^2\leq 2(\bar Q+\alpha)\|\nabla v\|_4\|q_e\|_4.
  \end{equation}
  Summing (\ref{l42})--(\ref{l44}) up, and integrating the resultant in $t$ yield
  \begin{align}
    &(\|U\|_4^4+\|T_e\|_4^2+\|q_e\|_4^2)(t)+2\int_0^t\big\||U|\nabla U\big\|_2^2 ds\nonumber\\
    \leq&\|U_0\|_4^4+\|T_{e,0}\|_4^2+\|q_{e,0}\|_4^2+2(1+\alpha)\int_0^t\|\nabla v\|_4(\|T_e\|_4+\|q_e\|_4)ds\nonumber\\
    &+C\int_0^t(1+\|U\|_4^4)(\|U\|_4^4+\|T_e\|_4^2+\|q_e\|_4^2)ds, \label{l45}
  \end{align}
  for $t\in[0,\mathcal T)$.

  We need to estimate the term $\int_0^t\|\nabla v\|_4(\|T_e\|_4+\|q_e\|_4)ds$ on the right-hand side of (\ref{l45}). To this end, applying Lemma \ref{lemapp2} (in the Appendix section) to equation (\ref{eq3}) yields
  \begin{align}
    \int_0^t\|\nabla v\|_4^4ds
    \leq&C\left[\|\nabla v_0\|_2^4+\left(\int_0^t\big\||U|\nabla U\big\|_2^2ds\right)^2+\int_0^t (\|T_e\|_4^4+\|q_e\|_4^4)ds\right], \label{l45-4}
  \end{align}
  for all $t\in[0,\mathcal T)$, where $C$ is a positive constant independent of $t$.
  Thanks to this estimate, it follows from the H\"older and Young inequalities that
  \begin{align}
    &2(1+\alpha)\int_0^t\|\nabla v\|_4(\|T_e\|_4+\|q_e\|_4)ds\nonumber\\
    \leq&Ct^{\frac12}\left(\int_0^t\|\nabla v\|_4^4ds\right)^{\frac14}\left(\int_0^t(\|T_e\|_4^4+\|q_e\|_4^4)ds \right)^{\frac14}\nonumber\\
    \leq&Ct^{\frac12}\left(\|\nabla v_0\|_2^4+\int_0^t\big\||U|\nabla U\big\|_2^2ds \right)^{\frac 12}\left(\int_0^t(\|T_e\|_4^4+\|q_e\|_4^4)ds \right)^{\frac14}\nonumber\\
    &+Ct^{\frac12}
    \left( \int_0^t (\|T_e\|_4^4+\|q_e\|_4^4)ds\right)^{\frac12}
    \nonumber\\
    \leq&\int_0^t\big\||U|\nabla U\big\|_2^2ds+C\left(\int_0^t(\|T_e\|_4^4+\|q_e\|_4^4)ds \right)^{\frac12}+C\|\nabla v_0\|_2^4.\label{l46}
  \end{align}

  Substituting (\ref{l46}) into (\ref{l45}), and denoting
  \begin{eqnarray*}
    &&f(t)=(\|U\|_4^4+\|T_e\|_4^2+\|q_e\|_4^2)(t)+\int_0^t\big\||U|\nabla U\big\|_2^2ds,
  \end{eqnarray*}
  we have
  $$
  f(t)\leq f(0)+C\|\nabla v_0\|_2^4+C\left(\int_0^tf^2(s)ds\right)^{\frac12}+C\int_0^t (1+\|U\|_4^4) f(s) ds,
  $$
  for all $t\in[0,\mathcal T)$.
  By Proposition \ref{basicest}, and using the Ladyzhenskaya inequality, one can easily check that $\int_0^\mathcal T(1+\|U\|_4^4)dt\leq C$, for a positive constant $C$ depending only on $\alpha,\bar Q,\mathcal T$ and the initial norm $\|(u_0, v_0, T_{e,0}, q_{e,0})\|_2$. Therefore, applying Lemma \ref{gron} to the above inequality, one obtains
  $$
  \sup_{0\leq t<\mathcal T}\|(U,T_e,q_e)(t)\|_4^2+\int_0^\mathcal T\big\||U|\nabla U\big\|_2^2dt\leq C,
  $$
  and further, recalling (\ref{l45-4}), proves the conclusion.
\end{proof}

Thanks to the {\it a priori} estimate stated in the above proposition, one can immediately obtain the $L^\infty(0,\mathcal T; H^1(\mathbb R^2))$ estimate on $u$ as stated in the following proposition.

\begin{proposition}
  \label{h1u}
  We have the following estimates
  \begin{equation*}
    \sup_{0\leq t<\mathcal T}\|\nabla u(t)\|_2^2+\int_0^\mathcal T\|\Delta u\|_2^2ds\leq C,
  \end{equation*}
  for a positive constant $C$ depending only on the parameters $\alpha, \bar Q, \mathcal T$ and the initial norm $\|u_0\|_{H^1(\mathbb R^2)}+\|(v_0,T_{e,0},q_{e,0})\|_{ L^2(\mathbb R^2)\cap L^4(\mathbb R^2)}$, and in particular is independent of $\varepsilon$.
\end{proposition}

\begin{proof}
  Multiplying equation (\ref{eq1}) by $-\Delta u$, and integrating over $\mathbb R^2$, then it follows from integration by parts that
  \begin{align*}
    \frac12\frac{d}{dt}\|\nabla u\|_2^2+\|\Delta u\|_2^2=\int_{\mathbb R^2}[(u\cdot\nabla)u+\nabla(v\otimes v)]\cdot\Delta u dxdy\\
    \leq 3\int_{\mathbb R^2}|U||\nabla U||\Delta u|dxdy\leq\frac12\|\Delta u\|_2^2+C\big\||U|\nabla U|\big\|_2^2,
  \end{align*}
  where, again, $U=(u,v)$, and thus
  $$
  \frac{d}{dt}\|\nabla u\|_2^2+\|\Delta u\|_2^2\leq C\big\||U|\nabla U|\big\|_2^2,
  $$
  for all $t\in[0,\mathcal T)$. From which, in view of Proposition \ref{l4est}, the conclusion follows.
\end{proof}

Finally, we are ready to prove the $L^\infty(0,\mathcal T; H^1(\mathbb R^2))$ estimate on $(v,T_e,q_e)$, that is the following proposition.

\begin{proposition}
  \label{h1v}
  The following estimate holds
  \begin{align*}
    \sup_{0\leq t<\mathcal T}\|(\nabla v,\nabla T_e,\nabla q_e)(t)\|_2^2+\int_0^\mathcal T\left(\|\Delta v\|_2^2 +\frac{\|\nabla q_e^+\|_2^2 }{\varepsilon}+\|\nabla u\|_\infty\right)dt\leq C,
  \end{align*}
  where $C$ is a positive constant depending only on $\alpha,\bar Q,\mathcal T$ and the initial norms $\|(u_0, v_0, T_{e,0}, q_{e,0})\|_{H^1}$, and in particular is independent of $\varepsilon$.
\end{proposition}

\begin{proof}
  Multiplying equation (\ref{eq4}) by $-\Delta T_e$, and integrating over $\mathbb R^2$, then it follows from integration by parts and the H\"older inequality that
  \begin{align*}
    \frac12\frac{d}{dt}\|\nabla T_e\|_2^2=&(1-\bar Q)\int_{\mathbb R^2}\nabla T_e\cdot\nabla(\nabla\cdot v) dxdy -\int_{\mathbb R^2}\partial_iu\cdot\nabla T_e\partial_i T_edxdy\\
    \leq&(1-\bar Q)\|\Delta v\|_2\|\nabla T_e\|_2 +\|\nabla u\|_\infty \|\nabla T_e\|_2^2.
  \end{align*}
  Similarly, one can derive from equation (\ref{eq5}) that
  $$
  \frac12\frac{d}{dt}\|\nabla q_e\|_2^2+\frac{1+\alpha}{\varepsilon}\|\nabla q_e^+\|_2^2\leq (\alpha+ \bar Q)\|\Delta v\|_2\|\nabla q_e\|_2+\|\nabla u\|_\infty \|\nabla q_e\|_2^2.
  $$
  Summing the previous two inequalities up yields
  \begin{align*}
  &\frac12\frac{d}{dt}(\|\nabla T_e\|_2^2+ \|\nabla q_e\|_2^2)+\frac{1+\alpha}{\varepsilon}\|\nabla q_e^+\|_2^2\\
  \leq&(1+\alpha)\|\Delta v\|_2(\|\nabla T_e\|_2+ \|\nabla q_e\|_2) +\|\nabla u\|_\infty (\|\nabla T_e\|_2^2+ \|\nabla q_e\|_2^2)\\
  \leq&\frac14\|\Delta v\|_2^2+[\|\nabla u\|_\infty+2(\alpha+1)^2](\|\nabla T_e\|_2^2+ \|\nabla q_e\|_2^2),
  \end{align*}
  and thus
  \begin{eqnarray}
   &&\frac{d}{dt}(\|\nabla T_e\|_2^2+ \|\nabla q_e\|_2^2) +\frac{1+\alpha}{\varepsilon}\|\nabla q_e^+\|_2^2\nonumber\\
  &\leq& 2[\|\nabla u\|_\infty+2(\alpha+1)^2](\|\nabla T_e\|_2^2+ \|\nabla q_e\|_2^2)+\frac12\|\Delta v\|_2^2.\label{LI1}
  \end{eqnarray}

  Multiplying equation (\ref{eq3}) by $-\Delta v$, and integrating over $\mathbb R^2$, then it follows from integration by parts and the Young inequality that
  \begin{align*}
    \frac12\frac{d}{dt}\|\nabla v\|_2^2+\|\Delta v\|_2^2=&\int_{\mathbb R^2}\left[\frac{1}{1+\alpha}\nabla(T_e-q_e)-(u\cdot\nabla)v-(v\cdot\nabla) u\right]\cdot\Delta vdxdy \\
    \leq&\frac14\|\Delta v\|_2^2+C\left(\|\nabla T_e\|_2^2+\|\nabla q_e\|_2^2+\big\||U|\nabla U\big\|_2^2\right),
  \end{align*}
  and thus
  \begin{equation}
    \frac{d}{dt}\|\nabla v\|_2^2+\frac32\|\Delta v\|_2^2\leq C\left(\|\nabla T_e\|_2^2+\|\nabla q_e\|_2^2+\big\||U|\nabla U\big\|_2^2\right).\label{LI2}
  \end{equation}

  Summing (\ref{LI1}) with (\ref{LI2}) up yields
  \begin{eqnarray*}
   &&\frac{d}{dt}\|(\nabla v, \nabla T_e, \nabla q_e)\|_2^2+\|\Delta v\|_2^2 +\frac{1+\alpha}{\varepsilon}\|\nabla q_e^+\|_2^2\nonumber\\
  &\leq&C\big\||U|\nabla U\big\|_2^2+C(\|\nabla u\|_\infty+1)(\|\nabla T_e\|_2^2+ \|\nabla q_e\|_2^2),
  \end{eqnarray*}
  from which, by the Gronwall inequality, and using Proposition \ref{l4est}, one obtains
  \begin{align}
    &\sup_{0\leq t<\mathcal T}\|(\nabla v, \nabla T_e, \nabla q_e)(t)\|_2^2+\int_0^\mathcal T\left(\|\Delta v\|_2^2 +\frac{1+\alpha}{\varepsilon}\|\nabla q_e^+\|_2^2\right)dt\nonumber\\
    \leq&e^{C\int_0^\mathcal T(\|\nabla u\|_\infty+1)dt}\left(\|(\nabla v_0, \nabla T_{e,0}, \nabla q_{e,0})\|_2^2+C\int_0^\mathcal T\big\||U|\nabla U\big\|_2^2dt\right)\nonumber\\
    \leq&C(\alpha,\bar Q,\mathcal T,\|(U_0,T_{e,0}, q_{e,0})\|_{H^1})\exp\left\{C\int_0^\mathcal T(\|\nabla u\|_\infty+1)dt\right\}.\label{LI3}
  \end{align}

  To complete the proof, one still need to estimate $\int_0^\mathcal T\|\nabla u\|_\infty dt$.
  It follows from Propositions \ref{l4est}--\ref{h1u} and the Ladyzhenskaya inequality that
  \begin{align}
  \int_0^\mathcal T(\|\nabla u\|_4^4+\|\nabla v\|_4^4)dt
    \leq& C\int_0^\mathcal T(\|\nabla u\|_2^2\|\Delta u\|_2^2+\|\nabla v\|_4^4)dt\nonumber\\
    \leq&C(\alpha,\bar Q, \mathcal T, \|(u_0, v_0,T_{e,0},q_{e,0})\|_{H^1}).\label{LI6}
  \end{align}
  We decompose $u$ as $u=\bar u+\hat u$, where $\bar u$ and $\hat u$, respectively, are the unique solutions to the following two systems
  \begin{equation}
    \label{LI4}
    \left\{
    \begin{array}{l}
    \partial_t\bar u-\Delta\bar u+\nabla\bar p=-(u\cdot\nabla)u-\nabla\cdot(v\otimes v),\\
    \nabla\cdot\bar u=0,\\
    \bar u|_{t=0}=0,
    \end{array}
    \right.
  \end{equation}
  and
  \begin{equation}
    \label{LI5}
    \left\{
    \begin{array}{l}
    \partial_t\hat u-\Delta\hat u+\nabla\hat p=0,\\
    \nabla\cdot\hat u=0,\\
    \hat u|_{t=0}=u_0.
    \end{array}
    \right.
  \end{equation}

  We are going to estimate $\bar u$ and $\hat u$. Let's first estimate $\bar u$. By the $L^q(0,\mathcal T; W^{2,q})$ type estimates for the Stokes equations (see, e.g., Solonnikov \cite{SOLONNIKOV1,SOLONNIKOV2}), we have
  $$
  \|(\partial_t\bar u, \Delta\bar u)\|_{L^q(\mathbb R^2\times(0,\mathcal T))}\leq C\||U|\nabla U\|_{L^q(\mathbb R^2\times(0,\mathcal T))},
  $$
  for any $q\in(1,\infty)$, and thus it follows from the H\"older inequality and Gagliardo-Nirenberg inequality, $\|\varphi\|_{12}^3\leq C\|\varphi\|_4^2\|\nabla\varphi\|_4$, (\ref{LI6}) and Proposition \ref{l4est} that
  \begin{align*}
    \int_0^\mathcal T\|\Delta\bar u\|_3^3dt\leq& C\int_0^\mathcal T\||U|\nabla U\|_3^3dt\leq C\int_0^\mathcal T\|\nabla U\|_4^3\|U\|_{12}^3dt\\
    \leq& C\int_0^\mathcal T\|\nabla U\|_4^4\|U\|_4^2dt\leq C(\alpha,\bar Q, \mathcal T, \|(u_0, v_0,T_{e,0},q_{e,0})\|_{H^1}).
  \end{align*}
  One can deduce easily from equation (\ref{LI4}), by using Proposition \ref{l4est}, that
  \begin{align*}
\sup_{0\leq t<\mathcal T}\|\nabla\bar u(t)\|_2^2&+\int_0^\mathcal T\|\Delta\bar u\|_2^2dt\leq C\int_0^\mathcal T\||U|\nabla U\|_2^2dt\\
  \leq& C(\alpha,\bar Q,\mathcal T,\|(u_0, v_0, T_{e,0}, q_{e, 0})\|_{H^1}).
  \end{align*}
  Thanks to the above two estimates, it follows from the Gagliardo-Nirenberg, $\|\varphi\|_\infty\leq C\|\varphi\|_2^{\frac14}\|\Delta\varphi\|_2^{\frac34}$, and the H\"older inequalities that
  \begin{align}
    \int_0^\mathcal T\|\nabla\bar u\|_\infty dt\leq& C\int_0^\mathcal T\|\nabla\bar u\|_2^{\frac14}\|\Delta\bar u\|_3^{\frac34}dt\nonumber\\
    \leq& C\left(\int_0^\mathcal T\|\nabla\bar u\|_2^2dt\right)^{\frac18}\left(\int_0^\mathcal T\|\Delta\bar u\|_3^3dt\right)^{\frac14}\mathcal T^{\frac85}\nonumber\\
    \leq& C(\alpha,\bar Q,\mathcal T,\|(u_0, v_0, T_{e,0}, q_{e, 0})\|_{H^1}).\label{LI7}
  \end{align}

  Next, we estimate $\hat u$. Multiplying equation (\ref{LI5}) by $(t\Delta^2-\Delta)\hat u$, and integrating the resultant over $\mathbb R^2$, then it follows from integration by parts that
  $$
  \frac12\frac{d}{dt}(\|\nabla\hat u\|_2^2+\|\sqrt t\Delta\hat u\|_2^2)+\frac12\|\Delta\hat u\|_2^2+\|\sqrt t\nabla\Delta\hat u\|_2^2=0.
  $$
  Therefore, we have
  $$
  \sup_{0\leq t<\mathcal T}(\|\nabla\hat u\|_2^2+\|\sqrt t\Delta\hat u\|_2^2)+\int_0^\mathcal T(\|\Delta\hat u\|_2^2+\|\sqrt t\nabla\Delta\hat u\|_2^2)dt\leq\|\nabla u_0\|_2^2.
  $$
  Thanks to this estimate, it follows from the Gagliardo-Nirenberg (Agmon), $\|\varphi\|_\infty\leq C\|\varphi\|_2^{\frac12}\|\Delta\varphi\|_2^{\frac12}$, and H\"older inequalities that
  \begin{align*}
    \int_0^\mathcal T\|\nabla\hat u\|_\infty dt\leq& C\int_0^\mathcal T\|\nabla\hat u\|_2^{\frac12}\|\nabla\Delta\hat u\|_2^{\frac12}dt
    =C\int_0^\mathcal T\|\nabla\hat u\|_2^{\frac12}\|\sqrt t\nabla\Delta\hat u\|_2^{\frac12}t^{-\frac14}dt\\
    \leq&C\left(\int_0^\mathcal T\|\nabla\hat u\|_2^2dt\right)^{\frac14}\left(\int_0^\mathcal T\|\sqrt t\nabla\Delta\hat u\|_2^2dt\right)^{\frac14}\left(\int_0^\mathcal Tt^{-\frac12}dt\right)^{\frac12}\\
    \leq&C\mathcal T^{\frac14}\|\nabla u_0\|_2^{\frac12}\|\nabla u_0\|_2^{\frac12}\mathcal T^{\frac14}=C\mathcal T^{\frac12}\|\nabla u_0\|_2.
  \end{align*}

  Combining the above estimate with (\ref{LI7}), one has
  $$
  \int_0^\mathcal T\|\nabla u\|_\infty dt\leq\int_0^\mathcal T(\|\nabla\bar u\|_\infty +\|\nabla\hat u\|_\infty) dt\leq C(\alpha,\bar Q,\mathcal T,\|(u_0, v_0, T_{e,0}, q_{e, 0})\|_{H^1}).
  $$
  which, when substituted into (\ref{LI3}), yields the conclusion.
\end{proof}

As a corollary of Propositions \ref{basicest}--\ref{h1v}, we have the {\it a priori} estimate to $(u, v, T_e, q_e)$, as stated in the following:

\begin{corollary}
  \label{cor}
Suppose that (\ref{req}) holds, and the initial data
\begin{equation}\label{asum2}
(u_0, v_0, T_{e,0}, q_{e,0})\in H^1(\mathbb R^2), \quad \nabla\cdot u_0=0.
\end{equation}
Let $(u, v, T_e, q_e)$ be the unique strong solution to system (\ref{eq1})--(\ref{eq5}), on $\mathbb R^2\times(0,\mathcal T)$, $0<\mathcal T<\infty$, with initial data $(u_0, v_0, T_{e,0}, q_{e,0})$. Then, the following hold:

(i) We have the estimate
\begin{align*}
\sup_{0\leq t<\mathcal T}&\|(u, v,T_e,q_e)(t)\|_{H^1}^2
+\int_0^\mathcal T\bigg(\frac{\|q_e^+\|_{H^1}^2}{\varepsilon}+
\|(u,v)\|_{H^2}^2 +\|\nabla u\|_\infty\bigg)dt \\ &+\int_0^\mathcal T\|(\partial_tu,\partial_tv,\partial_t T_e)\|_2^2  dt
\leq C\left(\alpha,\bar Q,\mathcal T,\|(u_0, v_0, T_{e,0}, q_{e,0})\|_{H^1}\right).
\end{align*}

(ii) Suppose in addition to (\ref{asum2}) that
$q_{e,0}^+=0, \text{ a.e.~on }\mathbb R^2,$
then we have
\begin{align*}
\sup_{0\leq t<\mathcal T} \frac{\|q_e^+\|_2^2}{\varepsilon}
 +\int_0^\mathcal T \|\partial_tq_e\|_2^2 dt\leq C\left(\alpha,\bar Q,\mathcal T,\|(u_0, v_0, T_{e,0}, q_{e,0})\|_{H^1}\right).
\end{align*}

(iii) Assume in addition to (\ref{asum2}) that $(\nabla T_{e,0},\nabla q_{e,0})\in L^m(\mathbb R^2)$, for some $m\in(2,\infty)$, then we
have the estimate
$$
\sup_{0\leq t<\mathcal T}\|(\nabla T_e,\nabla q_e)\|_m^2\leq C\left(\alpha,\bar Q,\mathcal T,m,\|(u_0, v_0, T_{e,0}, q_{e,0})\|_{H^1},\|(\nabla T_{e,0}, \nabla q_{e,0})\|_m\right).
$$
\end{corollary}

\begin{proof}
(i) The estimate on all the terms, except those involving the time derivatives, follow directly from Propositions \ref{basicest}--\ref{h1v}. The desired estimate for $(\partial_tu, \partial_tv)$ follows directly from the {\it a priori} estimate in Propositions \ref{l4est} and \ref{h1v}, by using the $L^2(0,T; H^2)$ type estimates to the Stokes and heat equations. By Propositions \ref{basicest}, \ref{h1u} and \ref{h1v}, it follows from equation (\ref{eq4}) and the Sobolev embedding inequalities that
\begin{align*}
  \int_0^\mathcal T\|\partial_tT_e\|_2^2dt\leq&\int_0^\mathcal T[(1-\bar Q)\|\nabla  v\|_2^2+\|u\|_\infty^2\|\nabla T_e\|_2^2]dt\\
  \leq& C+C\int_0^\mathcal T\|u\|_\infty^2dt\leq C+C\int_0^\mathcal T\|u\|_{H^2}^2dt\leq C.
\end{align*}

(ii) Multiplying equation (\ref{eq5}) by $\partial_tq_e$, and integrating over $\mathbb R^2$, then it follows from the Young and Sobolev embedding inequalities and Proposition \ref{h1v} that
\begin{eqnarray*}
  \frac{1+\alpha}{2\varepsilon}\frac{d}{dt}\|q_e^+\|_2^2 +\|\partial_tq_e\|_2^2
  &=&-\int_{\mathbb R^2}[u\cdot\nabla q_e+(\bar Q+\alpha)\nabla\cdot v]\partial_tq_edxdy\\
  &\leq&\frac12\|\partial_tq_e\|_2^2+C (\|u\|_{\infty}^2\|\nabla q_e\|_2^2+\|\nabla v\|_2^2)\\
  &\leq&\frac12\|\partial_tq_e\|_2^2+  C(\|u\|_{H^2}^2+1),
\end{eqnarray*}
from which, by (i), the conclusion in (ii) follows.

(iii) Applying the operator $\nabla$ to equation (\ref{eq4}), multiplying the resultant by $|\nabla T_e|^{m-2}\nabla T_e$, and integrating over $\mathbb R^2$, then it follows from integration by parts and the H\"older inequality that
  \begin{align*}
    \frac1m\frac{d}{dt}\|\nabla T_e\|_m^m=&(1-\bar Q)\int_{\mathbb R^2}|\nabla T_e|^{m-2}\nabla T_e\cdot\nabla(\nabla\cdot v) dxdy\\
    &-\int_{\mathbb R^2}\partial_iu\cdot\nabla T_e|\nabla T_e|^{m-2}\partial_i T_edxdy\\
    \leq&(1-\bar Q)\|\nabla^2v\|_m\|\nabla T_e\|_m^{m-1}+\|\nabla u\|_\infty \|\nabla T_e\|_m^m.
  \end{align*}
  Thus
  $$
  \frac{d}{dt}\|\nabla T_e\|_m\leq (1-\bar Q)\|\nabla^2v\|_m +\|\nabla u\|_\infty \|\nabla T_e\|_m.
  $$
  Similarly, one can derive from equation (\ref{eq5}) that
  $$
  \frac{d}{dt}\|\nabla q_e\|_m\leq (\alpha+ \bar Q)\|\nabla^2v\|_m+\|\nabla u\|_\infty \|\nabla q_e\|_m.
  $$

  Summing the above two inequalities, one obtains
  $$
  \frac{d}{dt}(\|\nabla T_e\|_m+ \|\nabla q_e\|_m)\leq (1+\alpha)\|\nabla^2v\|_m+\|\nabla u\|_\infty(\|\nabla T_e\|_m+ \|\nabla q_e\|_m),
  $$
  from which, integrating with respect to $t$, we have
  \begin{equation}
    \|(\nabla T_e,\nabla q_e)\|_m(t)\leq C\int_0^t\|\nabla^2v\|_mds+C\int_0^t\|\nabla u\|_\infty \|(\nabla T_e,\nabla q_e)\|_mds, \label{641}
  \end{equation}
  for all $t\in[0,\mathcal T)$.

  Applying Lemma \ref{lemapp3}, see the Appendix section below, to equation (\ref{eq3}), and using the Sobolev embedding inequality, one deduces
  \begin{align*}
    \int_0^t&\|\nabla^2v\|_mds\leq C\left[\|\nabla v_0\|_2+\left(\int_0^t\|(\nabla T_e, \nabla q_e,|u||\nabla v|,|v||\nabla u|)\|_m^2ds\right)^{\frac12}\right]\\
    \leq&C\left[\|\nabla v_0\|_2+\left( \int_0^t(\|(\nabla T_e, \nabla q_e)\|_m^2+\|u\|_{2m}^2\|\nabla v\|_{2m}^2+\|v\|_{2m}^2\|\nabla u\|_{2m}^2)ds\right)^{\frac12}\right]\\
    \leq&C\left[\|\nabla v_0\|_2+\left( \int_0^t(\|(\nabla T_e, \nabla q_e)\|_m^2+\|(u,v)\|_{H^1}^2\|(\nabla u, \nabla v)\|_{H^1}^2)ds\right)^{\frac12}\right],
  \end{align*}
  for any $t\in[0,\mathcal T)$, where $C$ is a positive constant depending only on $m$ and $\mathcal T$, and is in particular independent of $t\in[0,\mathcal T)$.
  By (i), the above inequality implies
  $$
  \int_0^t\|\nabla^2v\|_mds\leq C+C\left(\int_0^t\|(\nabla T_e, \nabla q_e)\|_m^2ds\right)^{\frac12},
  $$
  for any $t\in[0,\mathcal T)$, and for a positive constant $C$ independent of $t\in[0,\mathcal T)$. Substituting the above estimate into (\ref{641}), and setting $f(t)= \|(\nabla T_e,\nabla q_e)\|_m(t)$ yield
  \begin{equation*}
    f(t)\leq C\left[\|\nabla v_0\|_2+\left(\int_0^tf(s)^2ds\right)^{\frac12}+\int_0^t\|\nabla u\|_\infty f(s)ds\right],
  \end{equation*}
  for any $t\in[0,\mathcal T)$, where $C$ is a positive constant independent of $t\in[0,\mathcal T)$. Recalling (i), and applying Lemma \ref{gron}, the conclusion stated in (iii) follows.
\end{proof}

Now, we are ready to prove the global existence, uniqueness and well-posedness of strong solutions to the Cauchy problem of system (\ref{eq1})--(\ref{eq5}):

\begin{proof}[\textbf{Proof of Theorem \ref{glopositive}}]
  The uniqueness of strong solutions follows from Proposition \ref{loc} directly, while the {\it a priori} estimates in (i)--(iii) follow from (i)--(iii) of Corollary \ref{cor}, respectively. Therefore, we still need to prove the global existence of strong solutions as stated in (i), and the continuous dependence of the strong solutions on the initial date as stated in (iii).

  To prove the global existence of strong solutions, it suffices to extend the local solution established in Proposition \ref{loc} to be a global one. By repeating Proposition \ref{loc}, one can extend the local solution $(u, v, T_e, q_e)$ to the maximal interval of existence $[0,\mathcal T_*)$. Then, we need to show that $\mathcal T_*=\infty$. Suppose, by contradiction, that $\mathcal T_*<\infty$, then we must have
  \begin{equation*}
    \lim_{t\rightarrow \mathcal T_*^-} \|(u, v,T_e, q_e)\|_{H^1}^2 =\infty.
  \end{equation*}
  However, by Corollary \ref{cor}, which holds since $\mathcal T_*<\infty$, the quantity $\|(u, v,T_e, q_e)\|_{H^1}^2$ is bounded on $[0,\mathcal T_*)$, which is a contradiction, and thus $\mathcal T_*=\infty$.

  We now prove the continuous dependence of the unique strong solutions on the initial data as stated in (iii) on any finite interval $[0,\mathcal T]$. Therefore, we choose arbitrary $\mathcal T\in(0,\infty)$, and focus on the interval $[0,\mathcal T]$. Let $(u^{(1)}, v^{(1)}, T_e^{(1)}, q_e^{(1)})$ and $(u^{(2)}, v^{(2)}, T_e^{(2)}, q_e^{(2)})$ be the unique solutions to system (\ref{eq1})--(\ref{eq5}), respectively, with initial data $(u_0^{(1)}, v_0^{(1)}, T_{e,0}^{(1)}, q_{e,0}^{(1)})$ and $(u_0^{(2)}, v_0^{(2)}, T_{e,0}^{(2)}, q_{e,0}^{(2)})$. Denote by
  $$
  (\delta u, \delta v, \delta T_e, \delta q_e)=(u^{(1)}, v^{(1)}, T_e^{(1)}, q_e^{(1)})-(u^{(2)}, v^{(2)}, T_e^{(2)}, q_e^{(2)}),
  $$
  and
  $$
  (\delta u_0 , \delta v_0 , \delta T_{e,0} , \delta q_{e,0})=(u_0^{(1)}, v_0^{(1)}, T_{e,0}^{(1)}, q_{e,0}^{(1)})-(u_0^{(2)}, v_0^{(2)}, T_{e,0}^{(2)}, q_{e,0}^{(2)}).
  $$
  Then, similar to (\ref{add}), we have
  \begin{eqnarray}
  &&\frac{d}{dt}\|(\delta u,\delta v,\delta T_e,\delta q_e)\|_2^2+\frac12(\|\nabla\delta u\|_2^2+\|\nabla\delta v\|_2^2)\nonumber\\
  &\leq&C\int_{\mathbb R^2}[(|\nabla u^{(2)}|+|\nabla v^{(2)}|+|\nabla v^{(1)}|+|v^{(1)}|^2+|v^{(2)}|^2)(|\delta u|^2+|\delta v|^2)\nonumber\\
  &&+|\delta T_e|^2+|\delta q_e|^2+|\nabla T_e^{(2)}||\delta u||\delta T_e|+|\nabla  q_e^{(2)}||\delta u||\delta q_e|]dxdy,\label{add2}
  \end{eqnarray}
  for all $t\in(0,\mathcal T]$.
  All the integrals on the right-hand side of the above inequality, except the last two terms, can be dealt with in the way as before in (\ref{LINEW6.2}), while for the last two terms, we estimate them by the H\"older, Sobolev embedding and Young inequalities as follows
  \begin{align*}
  &C\int_{\mathbb R^2}(|\nabla T_e^{(2)}||\delta u||\delta T_e|+|\nabla  q_e^{(2)}||\delta u||\delta q_e|)dxdy\\
  \leq&C\|\nabla T_e^{(2)}\|_m\|\delta u\|_{\frac{2m}{m-2}}\|\delta T_e\|_2+C\|\nabla  q_e^{(2)}\|_m\|\delta u\|_{\frac{2m}{m-2}}\|\delta q_e\|_2\\
  \leq&C\|\nabla T_e^{(2)}\|_m\|\delta u\|_{H^1}\|\delta T_e\|_2+C\|\nabla  q_e^{(2)}\|_m\|\delta u\|_{H^1}\|\delta q_e\|_2\\
  \leq&\frac18\|\delta u\|_{H^1}^2+C(\|\nabla T_e^{(2)}\|_m^2+\|\nabla  q_e^{(2)}\|_m^2)(\|\delta T_e\|_2^2+\|\delta q_e\|_2^2),
  \end{align*}
  for all $t\in(0,\mathcal T]$. Therefore, we deduce from (\ref{add2}) that
  \begin{eqnarray*}
  &&\frac{d}{dt}\|(\delta u,\delta v,\delta T_e,\delta q_e)\|_2^2+\frac18\|(\delta u,\delta v)\|_{H^1}^2\nonumber\\
  &\leq& C\left(1+\|(u^{(2)}, v^{(2)})\|_4^4+\|(\nabla u^{(2)}, \nabla v^{(2)},\nabla v^{(1)})\|_2^2\right)\|(\delta u,\delta v,\delta T_e, \delta q_e)\|_2^2\nonumber\\
  &&+C(\|\nabla T_e^{(2)}\|_m^2+\|\nabla q_e^{(2)}\|_m^2)(\|\delta T_e\|_2^2+\|\delta q_e\|_2^2),
  \end{eqnarray*}
  for all $t\in(0,\mathcal T]$. Applying the Gronwall inequality to the above inequality yields
  \begin{eqnarray*}
    &&\sup_{0\leq s\leq t}\|(\delta u,\delta v,\delta T_e,\delta q_e)(s)\|_2^2+\frac18\int_0^t\|(\delta u,\delta v)\|_{H^1}^2ds\\
    &\leq&e^{
            C\int_0^t\left(1+\|(u^{(2)}, v^{(2)})\|_4^4+\|(\nabla u^{(2)}, \nabla v^{(2)},\nabla v^{(1)})\|_2^2+\|(\nabla T_e^{(2)},\nabla q_e^{(2)})\|_m^2 \right)ds}\\
            &&\times\|(\delta u_0 , \delta v_0 , \delta T_{e,0} , \delta q_{e,0})\|_2^2,
  \end{eqnarray*}
  for all $t\in(0,\mathcal T]$. Recalling the regularities in (i) and (iii), the above inequality implies the continuous dependence of the strong solution on the initial data on $[0,\mathcal T]$, for any arbitrary $\mathcal T\in(0,\infty)$.
  This completes the proof.
\end{proof}

\section{Global existence and uniqueness of the limiting system}
\label{globalzero}
In this section, we prove the global existence and uniqueness of strong solutions to the Cauchy problem of the limiting system (\ref{ineq1})--(\ref{ineq7}):

\begin{proof}[\textbf{Proof of Theorem \ref{glozero}}]
\textbf{(i) The global existence and regularities.} By Theorem \ref{glopositive}, for any positive $\varepsilon$, there is a unique global strong solution $(u_\varepsilon, v_\varepsilon, T_{e\varepsilon}, q_{e\varepsilon})$ to system (\ref{eq1})--(\ref{eq5}), with initial data $(u_0, v_0, T_{e,0}, q_{e,0})$, such that
\begin{align*}
  \sup_{0\leq t\leq\mathcal T}&\left(\frac{\|q_{e\varepsilon}^+(t)\|_2^2}{\varepsilon}+\|(u_\varepsilon, v_\varepsilon,T_{e\varepsilon},q_{e\varepsilon})(t)\|_{H^1}^2 \right)+\int_0^\mathcal T\bigg(\frac{\|\nabla q_{e\varepsilon}^+\|_2^2}{\varepsilon}+
  \|(u_\varepsilon,v_\varepsilon)\|_{H^2}^2\bigg)dt \\
  &+\int_0^\mathcal T\left(\|(\partial_tu_\varepsilon,\partial_tv_\varepsilon,\partial_tT_{e\varepsilon}, \partial_tq_{e\varepsilon})\|_2^2+\|\nabla u_\varepsilon\|_\infty\right)  dt\leq C,
\end{align*}
for any positive finite time $\mathcal T$, where $C$ is a constant depending only on $\alpha, \bar Q, \mathcal T$ and initial norms $\|(u_0, v_0,T_{e,0},q_{e,0})\|_{H^1}$, and in particular, is independent of $\varepsilon$. Moreover, if in addition that $(\nabla T_{e,0}, \nabla q_{e,0})\in L^m(\mathbb R^2)$, for some $m\in(2,\infty)$, then we have further that
$$
\sup_{0\leq t<\mathcal T}\|(\nabla T_{e\varepsilon},\nabla q_{e\varepsilon})(t)\|_m^2\leq C\left(\alpha,\bar Q,\mathcal T,m,\|(u_0, v_0, T_{e,0}, q_{e,0})\|_{H^1},\|(\nabla T_{e,0}, \nabla q_{e,0})\|_m\right),
$$
for any positive finite time $\mathcal T$, and, again, the estimate is independent of $\varepsilon$.

Thanks to the above $\varepsilon$-independent estimates, there is a subsequence, still denoted by $(u_\varepsilon, v_\varepsilon, T_{e\varepsilon}, q_{e\varepsilon})$, and $(u, v, T_{e}, q_{e})$, such that
\begin{eqnarray*}
&&(u_\varepsilon, v_\varepsilon)\overset{*}{\rightharpoonup}(u, v),\quad\mbox{ in }L^\infty(0,\mathcal T; H^1(\mathbb R^2)), \\
&&(u_\varepsilon, v_\varepsilon)\rightharpoonup(u, v),\quad\mbox{ in }L^2(0,\mathcal T;H^2(\mathbb R^2)), \\
&&(\partial_t u_\varepsilon, \partial_tv_\varepsilon)\rightharpoonup(\partial_tu, \partial_tv),\quad\mbox{ in }L^2(0,\mathcal T; L^2(\mathbb R^2)), \\
&&(T_{e\varepsilon}, q_{e\varepsilon})\overset{*}{\rightharpoonup}(T_e, q_e),\quad\mbox{ in }L^\infty(0,\mathcal T; H^1(\mathbb R^2)),\\
&&(\partial_tT_{e\varepsilon},\partial_tq_{e\varepsilon})
\rightharpoonup(\partial_tT_e,\partial_tq_e),\quad\mbox{ in }L^2(0,\mathcal T; L^2(\mathbb R^2)),\\
&&q_{e\varepsilon}^+\rightarrow0,\quad\mbox{ in }L^\infty(0,\mathcal T; L^2(\mathbb R^2))\cap L^2(0,T; H^1(\mathbb R^2)),
\end{eqnarray*}
for any positive finite time $\mathcal T$, where $\rightharpoonup$ and $\overset{*}{\rightharpoonup}$ are the weak and weak-* convergences, respectively.
The last convergence in the above implies that
\begin{equation*}
q_e^+=0,\mbox{ or equivalently }q_e\leq 0, \quad\mbox{ a.e.~in }\mathbb R^2\times(0,\mathcal T).
\end{equation*}
Moreover, by the Aubin-Lions lemma, and using the Cantor diagonal argument, we have a subsequence, still denoted by $(u_\varepsilon, v_\varepsilon, T_{e\varepsilon}, q_{e\varepsilon})$, such that
\begin{eqnarray*}
  &(u_\varepsilon, v_\varepsilon)\rightarrow(u, v),\quad \mbox{ in }C([0,\mathcal T]; L^2(B_R))\cap L^2(0,T; H^1(B_R)),\\
  &(T_{e\varepsilon}, q_{e\varepsilon})\rightarrow(T_e, q_e),\quad\mbox{ in }C([0,\mathcal T]; L^2({B_R})),
\end{eqnarray*}
for any positive finite  time $\mathcal T$, and  disc $B_R \subset  \mathbb{R}^2$, of arbitrary radius $R>0$.

Thanks to the previous convergences, one can take the limit $\varepsilon\rightarrow0^+$ in the equations (\ref{eq1})--(\ref{eq4}) for $(u_\varepsilon, v_\varepsilon, T_{e\varepsilon}, q_{e\varepsilon})$ to deduce that $(u, v, T_e, q_e)$ satisfies equations (\ref{eq1})--(\ref{eq4}), a.e.~in $\mathbb R^2\times(0,\infty)$, since $R$ in the previous strong convergences is arbitrary; and moreover, by the lower semi-continuity of the norms, the {\it a priori} estimates stated in Theorem \ref{glozero} hold.
In order to complete the proof of existence, we still need to prove that $q_e$ satisfies inequalities (\ref{ineq5})--(\ref{ineq7}). Inequality (\ref{ineq6}) has already been verified before. While for (\ref{ineq5}), note that equation (\ref{eq5}) for $q_{e\varepsilon}$ implies that
$$
\partial_tq_{e\varepsilon}+u_\varepsilon\cdot\nabla q_{e\varepsilon}+(\bar Q+\alpha)\nabla\cdot v_\varepsilon\leq 0, \quad\mbox{a.e.~in }\mathbb R^2\times(0,\infty),
$$
from which, recalling the previous convergences, one can take the limit $\varepsilon\rightarrow0^+$ to see that
\begin{equation*}
\partial_tq_{e }+u \cdot\nabla q_{e }+(\bar Q+\alpha)\nabla\cdot v \leq 0, \quad\mbox{a.e.~in }\mathbb R^2\times(0,\infty),
\end{equation*}
which is (\ref{ineq5}).

It remains to verify (\ref{ineq7}). To this end, let's define the set
$$
\mathcal O^-=\{(x,t)|q_e(x,t)<0, x\in\mathbb R^2, t\in(0,\infty)\},
$$
and for any positive integers $j,k,l$, we define
$$
\mathcal O^-_{jkl}=\left\{(x,t)\bigg|q_e(x,t)<-\frac1j, x\in B_k, t\in(0,l)\right\},
$$
where $B_k \subset \mathbb{R}^2$ is a disc of radius $k$, and  $j,k,l \in \mathbb{N}$.
Noticing that
$$
\mathcal O^-=\cup_{j}^\infty\cup_{k=1}^\infty\cup_{l=1}^\infty \mathcal O^-_{jkl},
$$
to prove that (\ref{ineq7}) holds a.e.~on $\mathcal O^-$, it suffices to show that it holds a.e.~on $\mathcal O^-_{jkl}$, for any positive integers $j,k,l$.
Now, let's fix the positive integers $j,k,l$. Recalling that $q_{e\varepsilon}\rightarrow q_e$ in $C([0,\mathcal T]; L^2(B_R))$, for any positive time $\mathcal T$ and positive radius $R$, it is straightforward that $q_{e\varepsilon}\rightarrow q_e$ in $L^2(\Omega_{jkl})$. Therefore, there is a subsequence, still denoted by $q_{e\varepsilon}$, such that $q_{e\varepsilon}\rightarrow q_e$, a.e.~on $\mathcal O^-_{jkl}$. By the Egoroff theorem, for any positive number $\eta>0$, there is a subset $E_\eta$ of $\mathcal O^-_{jkl}$, with $|E_\eta|\leq\eta$, such that
$$
q_{e\varepsilon}\rightarrow q_e,\quad\mbox{uniformly on }\mathcal O^-_{jkl}\setminus E_\eta.
$$
Recalling the definition of $\mathcal O^-_{jkl}$, this implies that for sufficiently small positive $\varepsilon$, it holds that
$$
q_{e\varepsilon}\leq q_e+\frac{1}{2j}\leq-\frac{1}{2j}<0, \quad \mbox{on }\mathcal O^-_{jkl}\setminus E_\eta.
$$
As a result, by equation (\ref{eq5}) for $q_{e\varepsilon}$, we have, for any sufficiently small positive $\varepsilon$, that
$$
\mathcal G_\varepsilon:=\partial_tq_{e\varepsilon}+u_\varepsilon\cdot\nabla q_{e\varepsilon}+(\bar Q+\alpha)\nabla\cdot v_\varepsilon=0, \quad\mbox{ a.e.~on }\mathcal O^-_{jkl}\setminus E_\eta.
$$
Noticing that
$$
\mathcal G_\varepsilon\rightharpoonup  \partial_tq_{e }+u \cdot\nabla q_{e }+(\bar Q+\alpha)\nabla\cdot v=:\mathcal G ,\quad\mbox{ in }L^2(0,\mathcal T; L^2(\mathbb R^2)),
$$
for any positive finite time $\mathcal T$, which in particular implies $\mathcal G_\varepsilon\rightharpoonup\mathcal G$, in $L^2(\mathcal O_{jkl}\setminus E_\eta)$. Since $\mathcal G_\varepsilon=0$, a.e.~on $\mathcal O_{jkl}\setminus E_\eta$, we have $\mathcal G=0$, a.e.~on $\mathcal O_{jkl}\setminus E_\eta$, that is
$$
\partial_tq_{e }+u \cdot\nabla q_{e }+(\bar Q+\alpha)\nabla\cdot v=0, \quad\mbox{ a.e.~on }\Omega_{jkl}\setminus E_\eta.
$$
By Lemma \ref{zeroa.e.}, this implies that the above equation holds, a.e.~on $\mathcal O^-_{jkl}$, and further on $\mathcal O^-$, in other words, (\ref{ineq7}) holds.

Therefore, $(u, v, T_e, q_e)$ is a global strong solution to system (\ref{ineq1})--(\ref{ineq7}), with initial data $(u_0, v_0, T_{e,0}, q_{e,0})$, satisfying the regularities stated in the theorem.

\textbf{(ii) The uniqueness.} Let $(u, v, T_e, q_e)$ and $(\tilde u, \tilde v, \tilde T_e, \tilde q_e)$ be two strong solutions to system (\ref{ineq1})--(\ref{ineq7}), with the same initial data $(u_0, v_0, T_{e,0}, q_{e, 0})$. Define the new functions
$$
(\delta u, \delta v, \delta T_e,\delta q_e)=(u, v, T_e, q_e)-(\tilde u, \tilde v, \tilde T_e, \tilde q_e).
$$
Then, one can easily check that $(\delta u, \delta v, \delta T_e, \delta q_e)$ satisfies equations (\ref{deq1})--(\ref{deq4}), and the same argument as that for (\ref{uni1}) yields
\begin{align}
  &\frac{d}{dt}\|(\delta u,\delta v,\delta T_e)\|_2^2+\|\nabla\delta u\|_2^2+\|\nabla\delta v\|_2^2\nonumber\\
  \leq&C\int_{\mathbb R^2}[(|\nabla\tilde u|+|\nabla\tilde v|+|\nabla v|+|v|^2+|\tilde v|^2)(|\delta u|^2+|\delta v|^2)\nonumber\\
  &+|\delta T_e|^2+|\delta q_e|^2+|\nabla\tilde T_e||\delta u||\delta T_e|]dxdy.\label{uni1'}
\end{align}

We need to estimate $\delta q_e$. To this end, we first derive the equation for $\delta q_e$. We divide the domain $\Omega:=\mathbb R^2\times(0,\infty)$ as follows
$$
\Omega=\Omega_1\cup \Omega_2\cup \Omega_3\cup \Omega_4,
$$
where
\begin{eqnarray*}
  &\Omega_1=\{q_e<0\}\cap\{\tilde q_e<0\}, \quad \Omega_2=\{q_e<0\}\cap\{\tilde q_e=0\},\\
  &\Omega_3=\{q_e=0\}\cap\{\tilde q_e<0\},\quad \Omega_4=\{q_e=0\}\cap\{\tilde q_e=0\}.
\end{eqnarray*}
On the set $\Omega_1$, $q_e$ and $\tilde q_e$ satisfies, respectively
\begin{eqnarray*}
  &\partial_tq_e+u\cdot\nabla q_e+(\bar Q+\alpha)\nabla\cdot v=0,\\
  &\partial_t\tilde q_e+\tilde u\cdot\nabla\tilde  q_e+(\bar Q+\alpha)\nabla\cdot\tilde  v=0.
\end{eqnarray*}
Subtracting the above two equations yields
\begin{equation}
  \label{neq1}
  \partial_t\delta q_e+u\cdot\nabla\delta q_e+\delta u\cdot\nabla\tilde q_e+(\bar Q+\alpha)\nabla\cdot\delta v=0,\quad\mbox{ on }\Omega_1.
\end{equation}
On the set $\Omega_2$, $q_e$ satisfies
$$
\partial_tq_e+u\cdot\nabla q_e+(\bar Q+\alpha)\nabla\cdot v=0,
$$
while for $\tilde q_e$, since $\tilde q_e\equiv0$ on $\Omega_2$, one has $(\partial_tq_e,\nabla q_e)=0$, a.e.~on $\Omega_2$, and thus $\partial_t\tilde q_e+\tilde u\cdot\nabla\tilde q_e=0$, a.e.~on $\Omega_2$. Here, we have used the well-known fact that the derivatives of a function $f\in W^{1,1}_{\text{loc}}(\Omega)$ vanish, a.e.~on any level set $\{(x,y,t)\in\Omega|f(x,y,t)=c\}$, see, e.g., \cite{EVANS} or page 297 of \cite{GIOVANNI}. We will used, without any further mentions, this fact several times in the proof of this part. Therefore, one has
\begin{equation}
  \partial_t\delta q_e+u\cdot\nabla\delta q_e+\delta u\cdot\nabla\tilde q_e+(\bar Q+\alpha)\nabla\cdot v =0,\quad\mbox{a.e.~on }\Omega_2.\label{neq2}
\end{equation}
Similar to (\ref{neq2}), on the domain $\Omega_3$, one has
\begin{equation}
  \label{neq3}
  \partial_t\delta q_e+u\cdot\nabla\delta q_e+\delta u\cdot\nabla\tilde q_e-(\bar Q+\alpha)\nabla\cdot\tilde v=0,\quad\mbox{a.e.~on }\Omega_3.
\end{equation}
Finally, since $\tilde q_e=q_e=0$, on $\Omega_4$, one has
$$
\partial_t\delta q_e+u\cdot\nabla\delta q_e+\delta u\cdot\nabla\tilde q_e=0,\quad\mbox{a.e.~on }\Omega_4.
$$
Thanks to the last equation, as well as (\ref{neq1})--(\ref{neq3}), we obtain the equation for $\delta q_e$ as
\begin{eqnarray}
&\partial_t\delta q_e+u\cdot\nabla\delta q_e+\delta u\cdot\nabla\tilde q_e
  = -(\bar Q+\alpha)[\nabla\cdot\delta v\chi_{\Omega_1}
  +\nabla\cdot v\chi_{\Omega_2}-\nabla\cdot\tilde v\chi_{\Omega_3}]\nonumber\\
  &=-(\bar Q+\alpha)[\nabla\cdot\delta v-\nabla\cdot\delta v\chi_{\Omega_4}
  +\nabla\cdot \tilde v\chi_{\Omega_2}-\nabla\cdot v\chi_{\Omega_3}],\label{deq5}
\end{eqnarray}
a.e.~on $\Omega=\mathbb R^2\times(0,\infty).$ Moreover, equation (\ref{deq5}) holds in $L^2_{\text{loc}}([0,\infty); L^2(\mathbb R^2))$.

Multiplying equation (\ref{deq5}) by $\delta q_e$, and integrating over $\mathbb R^2$, then it follows from integration by parts that
\begin{align}
  \frac12\frac{d}{dt}\|\delta q_e\|_2^2=&-\int_{\mathbb R^2}[\delta u\cdot\nabla\tilde q_e\delta q_e+(\bar Q+\alpha)\nabla\cdot\delta v] \delta q_e dxdy\nonumber\\
  &-(\bar Q+\alpha)\int_{\mathbb R^2}(\nabla\cdot\tilde v\chi_{\Omega_2}-\nabla\cdot v\chi_{\Omega_3})(q_e-\tilde q_e)dxdy\nonumber\\
  \leq&\frac14\int_{\mathbb R^2}|\nabla\delta v|^2dxdy+C\int_{\mathbb R^2}(|\delta q_e|^2+|\nabla\tilde q_e||\delta u||\delta q_e|)dxdy\nonumber\\
  &-(\bar Q+\alpha)\int_{\mathbb R^2}(\nabla\cdot\tilde v\chi_{\Omega_2}-\nabla\cdot v\chi_{\Omega_3})(q_e-\tilde q_e)dxdy. \label{uni2}
\end{align}
Recalling that $\tilde q_e=0$ on $\Omega_2$, we have $\partial_t\tilde q_e+\tilde u\cdot\nabla\tilde q_e=0$, a.e.~on $\Omega_2$, and thus it follows from (\ref{ineq5}) for $(\tilde u, \tilde v, \tilde T_e, \tilde q_e)$
that $\nabla\cdot\tilde v\leq0$, a.e.~on $\Omega_2$. Similarly, one has $\nabla\cdot v\leq0$, a.e.~on $\Omega_3$. Thanks to these facts, we deduce
\begin{eqnarray*}
  \nabla\cdot\tilde v\chi_{\Omega_2}(q_e-\tilde q_e)=\nabla\cdot\tilde v\chi_{\Omega_2}q_e\geq0,\\
  -\nabla\cdot v\chi_{\Omega_3}(q_e-\tilde q_e)=\nabla\cdot v\chi_{\Omega_3}\tilde q_e\geq0.
\end{eqnarray*}
Therefore, it follows from (\ref{uni2}) that
\begin{equation*}
  \frac{d}{dt}\|\delta q_e\|_2^2\leq\frac12\|\nabla\delta v\|_2^2+C\int_{\mathbb R^2}(|\delta q_e|^2+|\nabla\tilde q_e||\delta u||\delta q_e|)dxdy.
\end{equation*}

Summing the above inequality with (\ref{uni1'}) yields
\begin{eqnarray}
  &&\frac{d}{dt}\|(\delta u,\delta v,\delta T_e,\delta q_e)\|_2^2+\frac12(\|\nabla\delta u\|_2^2+\|\nabla\delta v\|_2^2)\nonumber\\
  &\leq&C\int_{\mathbb R^2}[(|\nabla\tilde u|+|\nabla\tilde v|+|\nabla v|+|v|^2+|\tilde v|^2)(|\delta u|^2+|\delta v|^2)\nonumber\\
  &&+|\delta T_e|^2+|\delta q_e|^2+|\nabla\tilde T_e||\delta u||\delta T_e|+|\nabla\tilde q_e||\delta u||\delta q_e|]dxdy,\label{add3}
\end{eqnarray}
which is exactly the same as inequality (\ref{add}), from which, by the same argument as that in the proof of the uniqueness part of Proposition \ref{loc}, one obtains
$$\|(\delta u,\delta v,\delta T_e,\delta q_e)\|_2^2\equiv0.$$
This proves the uniqueness.

\textbf{(iii) Continuous dependence.}
Let $(u^{(i)}, v^{(i)}, T_e^{(i)}, q_e^{(i)})$ be the unique solutions to system (\ref{ineq1})--(\ref{ineq7}), with initial data $(u_0^{(i)}, v_0^{(i)}, T_{e,0}^{(i)}, q_{e,0}^{(i)})$, $i=1,2$. Suppose, in addition that $(\nabla T_{e,0}^{(i)}, \nabla q_{e,0}^{(i)})\in L^m(\mathbb R^2)$, for some $m\in(2,\infty)$. Then, recalling what we have proven in (i), $(u^{(i)}, v^{(i)}, T_e^{(i)}, q_e^{(i)})$ has the additional regularity that $(T_e^{(i)}, q_e^{(i)})\in L^\infty(0,\mathcal T; L^m(\mathbb R^2))$, for any positive time $\mathcal T$.

Denote by
  $$
  (\delta u, \delta v, \delta T_e, \delta q_e)=(u^{(1)}, v^{(1)}, T_e^{(1)}, q_e^{(1)})-(u^{(2)}, v^{(2)}, T_e^{(2)}, q_e^{(2)}),
  $$
  and
  $$
  (\delta u_0 , \delta v_0 , \delta T_{e,0} , \delta q_{e,0})=(u_0^{(1)}, v_0^{(1)}, T_{e,0}^{(1)}, q_{e,0}^{(1)})-(u_0^{(2)}, v_0^{(2)}, T_{e,0}^{(2)}, q_{e,0}^{(2)}).
  $$
  Then, similar to (\ref{add3}), we have
  \begin{eqnarray*}
  &&\frac{d}{dt}\|(\delta u,\delta v,\delta T_e,\delta q_e)\|_2^2+\frac12(\|\nabla\delta u\|_2^2+\|\nabla\delta v\|_2^2)\nonumber\\
  &\leq&C\int_{\mathbb R^2}[(|\nabla u^{(2)}|+|\nabla v^{(2)}|+|\nabla v^{(1)}|+|v^{(1)}|^2+|v^{(2)}|^2)(|\delta u|^2+|\delta v|^2)\nonumber\\
  &&+|\delta T_e|^2+|\delta q_e|^2+|\nabla T_e^{(2)}||\delta u||\delta T_e|+|\nabla  q_e^{(2)}||\delta u||\delta q_e|]dxdy,
  \end{eqnarray*}
  which is exactly of the same form as (\ref{add2}). Therefore, by the same argument as that in the proof of the continuous dependence part of (iii) of Theorem \ref{glopositive}, we obtain
  \begin{eqnarray*}
    &&\sup_{0\leq s\leq t}\|(\delta u,\delta v,\delta T_e,\delta q_e)(s)\|_2^2+\frac18\int_0^t\|(\delta u,\delta v)\|_{H^1}^2ds\\
    &\leq&e^{
            C\int_0^t\left(1+\|(u^{(2)}, v^{(2)})\|_4^4+\|(\nabla u^{(2)}, \nabla v^{(2)},\nabla v^{(1)})\|_2^2+\|(\nabla T_e^{(2)},\nabla q_e^{(2)})\|_m^2 \right)ds}\\
            &&\times\|(\delta u_0 , \delta v_0 , \delta T_{e,0} , \delta q_{e,0})\|_2^2.
  \end{eqnarray*}
  Recalling the regularities of $(u^{(i)}, v^{(i)}, T_e^{(i)}, q_e^{(i)})$, $i=1,2$, the above inequality implies the continuous dependence of strong solutions on the initial data. This completes the proof of Theorem \ref{glozero}.
\end{proof}

\section{Strong convergence of the relaxation limit}
\label{convergence}
In this section, we prove the strong convergence of the relaxation limit, as $\varepsilon\rightarrow0^+$, of system (\ref{eq1})--(\ref{eq5}) to the limiting system (\ref{ineq1})--(\ref{ineq7}):

\begin{proof}[\textbf{Proof of Theorem \ref{cong}}]
Define the difference function $(\delta u_\varepsilon, \delta v_\varepsilon, \delta T_{e\varepsilon}, \delta q_{e\varepsilon})$ as
$$
(\delta u_\varepsilon, \delta v_\varepsilon, \delta T_{e\varepsilon}, \delta q_{e\varepsilon})=(u_\varepsilon,v_\varepsilon, T_{e\varepsilon}, q_{e\varepsilon})-(u, v, T_e, q_e).
$$
Taking the subtraction between equations (\ref{eq1})--(\ref{eq4}), for $(u_\varepsilon, v_\varepsilon, T_{e\varepsilon}, q_{e\varepsilon})$, and equations (\ref{ineq1})--(\ref{ineq4}), for $(u, v, T_e, q_e)$, one can easily check that
\begin{align}
\partial_t\delta u_\varepsilon+(\delta u_\varepsilon\cdot\nabla)\delta u_\varepsilon&+(\delta u_\varepsilon\cdot\nabla)u+(u\cdot\nabla)\delta u_\varepsilon-\Delta\delta u_\varepsilon\nonumber\\
&+\nabla\delta p_\varepsilon+\nabla\cdot(\delta v_\varepsilon
  \otimes\delta v_\varepsilon+\delta v_\varepsilon\otimes v+v\otimes\delta v_\varepsilon)=0,\label{diff1}\\
\nabla\cdot\delta& u_\varepsilon=0,\label{diff2}\\
\partial_t\delta v_\varepsilon+(\delta u_\varepsilon\cdot\nabla)\delta v_\varepsilon&+(\delta u_\varepsilon\cdot\nabla)v+(u\cdot\nabla)\delta v_\varepsilon-\Delta\delta v_\varepsilon+(\delta v_\varepsilon\cdot\nabla)\delta u_\varepsilon\nonumber\\
  &+(\delta v_\varepsilon\cdot\nabla)u+(v\cdot\nabla)\delta u_\varepsilon=\frac{1}{1+\alpha}\nabla(\delta T_{e\varepsilon} -\delta q_{e\varepsilon}), \label{diff3}\\
\partial_t\delta T_{e\varepsilon}+\delta u_\varepsilon\cdot\nabla\delta T_{e\varepsilon}&+\delta u_\varepsilon\cdot\nabla T_e+u\cdot\nabla\delta T_{e\varepsilon}-(1-\bar Q)\nabla\cdot\delta v_\varepsilon=0,\label{diff4}
\end{align}
where (\ref{diff1})--(\ref{diff4}) hold a.e.~on $\mathbb R^2\times(0,\infty)$ and in $L^2_{\text{loc}}([0,\infty); L^2(\mathbb R^2))$.

Multiplying equations (\ref{diff1}), (\ref{diff3}) and (\ref{diff4}) by $\delta u_\varepsilon$, $\delta v_\varepsilon$ and $\delta T_{e\varepsilon}$, respectively, summing the resultants, integrating over $\mathbb R^2$, and noticing that
\begin{equation*}
  \int_{\mathbb R^2}[\nabla\cdot(\delta v_\varepsilon\otimes\delta v_\varepsilon)\cdot\delta u_\varepsilon+(\delta v_\varepsilon\cdot\nabla)\delta u_\varepsilon\cdot\delta v_\varepsilon]dxdy=0,
\end{equation*}
it follows from integration by parts that
\begin{eqnarray*}
  &&\frac12\frac{d}{dt}\|(\delta u_\varepsilon,\delta v_\varepsilon,\delta T_{e\varepsilon})\|_2^2+\|(\nabla\delta u_\varepsilon,\nabla\delta v_\varepsilon)\|_2^2\\
  &=&-\int_{\mathbb R^2}[(\delta u_\varepsilon\cdot\nabla)u+\nabla\cdot(\delta v_\varepsilon\otimes v+v\otimes\delta v_\varepsilon)]\cdot\delta u_\varepsilon dxdy\\
  &&-\int_{\mathbb R^2}[(\delta u_\varepsilon\cdot\nabla)v+(\delta v_\varepsilon\cdot\nabla)u+(v\cdot\nabla)\delta u_\varepsilon]\cdot\delta v_\varepsilon dxdy\\
  &&-\frac{1}{1+\alpha}\int_{\mathbb R^2}(\nabla\cdot\delta v_\varepsilon)(\delta T_{e\varepsilon}-\delta q_{e\varepsilon})  dxdy\\
  &&-\int_{\mathbb R^2}[\delta u_\varepsilon\cdot\nabla T_e-(1-\bar Q)\nabla\cdot\delta v_\varepsilon]\delta T_{e\varepsilon}dxdy,
\end{eqnarray*}
from which, by the Young inequality, we deduce
\begin{eqnarray*}
  &&\frac12\frac{d}{dt}\|(\delta u_\varepsilon,\delta v_\varepsilon,\delta T_{e\varepsilon})\|_2^2+\|(\nabla\delta u_\varepsilon,\nabla\delta v_\varepsilon)\|_2^2\\
  &\leq&\int_{\mathbb R^2}[(|\nabla u||\delta u_\varepsilon|+2|\nabla v||\delta v_\varepsilon|+2|v||\nabla\delta v_\varepsilon|)|\delta u_\varepsilon|+(|\nabla v||\delta u_\varepsilon|\\
  &&+|\nabla u||\delta v_\varepsilon|+|v||\nabla\delta u_\varepsilon|)|\delta v_\varepsilon|]dxdy +\frac{1}{1+\alpha}\int_{\mathbb R^2}|\nabla\delta v_\varepsilon|(|\delta T_{e\varepsilon}|+|\delta q_{e\varepsilon}|)dxdy\\
  &&+\int_{\mathbb R^2}[(1-\bar Q)|\nabla\delta v_\varepsilon||\delta T_{e\varepsilon}|+|\nabla T_e||\delta u_\varepsilon||\delta T_{e\varepsilon}|]dxdy\\
  &\leq&\frac12\int_{\mathbb R^2}(|\nabla\delta u_\varepsilon|^2+|\nabla\delta v_\varepsilon|^2)dxdy+C\int_{\mathbb R^2}[(|\nabla u|+|\nabla v|+|v|^2)\\
  &&\times(|\delta u_\varepsilon|^2+|\delta v_\varepsilon|^2)+|\delta T_{e\varepsilon}|^2+|\delta q_{e\varepsilon}|^2+|\nabla T_e||\delta u_\varepsilon||\delta T_{e\varepsilon}|]dxdy.
\end{eqnarray*}
Therefore, we obtain
\begin{eqnarray}
  && \frac{d}{dt}\|(\delta u_\varepsilon,\delta v_\varepsilon,\delta T_{e\varepsilon})\|_2^2+\|(\nabla\delta u_\varepsilon,\nabla\delta v_\varepsilon)\|_2^2\nonumber\\
  &\leq&C\int_{\mathbb R^2}[(|\nabla u|+|\nabla v|+|v|^2)(|\delta u_\varepsilon|^2+|\delta v_\varepsilon|^2)\nonumber\\
  && +|\delta T_{e\varepsilon}|^2+|\delta q_{e\varepsilon}|^2+|\nabla T_e||\delta u_\varepsilon||\delta T_{e\varepsilon}|]dxdy.\label{con1}
\end{eqnarray}

We still need to estimate $\|\delta q_{e\varepsilon}\|_2^2$. To this end, we first derive the equation for $\delta q_{e\varepsilon}$. On the set $\{(x,y,t)\in\mathbb R^2\times(0,\infty)|q_e(x,y,t)<0\}$, $q_{e\varepsilon}$ and $q_e$ satisfy equations (\ref{eq5}) and (\ref{ineq7}), respectively, and thus $\delta q_{e\varepsilon}$ satisfies
\begin{equation*}
  \partial_t\delta q_{e\varepsilon}+\delta u_\varepsilon\cdot\nabla\delta q_{e\varepsilon} +\delta u_\varepsilon\cdot\nabla q_e+u\cdot\nabla\delta q_{e\varepsilon}
   +(\bar Q+\alpha)\nabla\cdot\delta v_\varepsilon=-\frac{1+\alpha}{\varepsilon}q_{e\varepsilon}^+,
\end{equation*}
a.e.~on $\{(x,y,t)\in\mathbb R^2\times(0,\infty)|q_e(x,y,t)<0\}$.
On the set $\mathcal O:=\{(x,t)\in\mathbb R^2\times(0,\infty)|q_e(x,t)=0\}$, recalling, again, the well-known fact that the derivatives of a function $f\in W^{1,1}_{\text{loc}}(\mathbb R^2\times(0,\infty)$ vanish, a.e.~on any level set $\{(x,y,t)\in\mathbb R^2\times(0,\infty)|f(x,y,t)=c\}$, we have
$\partial_tq_e+u\cdot\nabla q_e=0$, a.e.~on $\mathcal O$, and $q_{e\varepsilon}$ satisfies (\ref{eq5}). Consequently, $\delta q_{e\varepsilon}$ satisfies
\begin{equation*}
  \partial_t\delta q_{e\varepsilon}+\delta u_\varepsilon\cdot\nabla\delta q_{e\varepsilon} +\delta u_\varepsilon\cdot\nabla q_e+u\cdot\nabla\delta q_{e\varepsilon}+(\bar Q+\alpha)\nabla\cdot v_\varepsilon
  = -\frac{1+\alpha}{\varepsilon}q_{e\varepsilon}^+,
\end{equation*}
a.e.~on $\mathcal O$. Combing the above two equations, one can see that $\delta q_{e\varepsilon}$ satisfies
\begin{align}
  \partial_t\delta q_{e\varepsilon}&+\delta u_\varepsilon\cdot\nabla\delta q_{e\varepsilon}+\delta u_\varepsilon\cdot\nabla q_e+u\cdot\nabla\delta q_{e\varepsilon}\nonumber\\
  &+(\bar Q+\alpha)\nabla\cdot\delta v_\varepsilon=-\frac{1+\alpha}{\varepsilon}q_{e\varepsilon}^+-(\bar Q+\alpha)\nabla\cdot v\chi_{\mathcal O}(x,y,t),  \label{con2}
\end{align}
a.e.~on $\mathbb R^2\times(0,\infty)$ and in $L^2_{\text{loc}}([0,\infty); L^2(\mathbb R^2))$.

Multiplying equation (\ref{con2}) by $\delta q_{e\varepsilon}$, and integrating over $\mathbb R^2$, then it follows from integration by parts that \begin{align}
  &\frac12\frac{d}{dt}\|\delta q_{e\varepsilon}\|_2^2+\frac{1+\alpha}{\varepsilon}\int_{\mathbb R^2}q_{e\varepsilon}^+\delta q_{e\varepsilon} dxdy\nonumber\\
  =&-\int_{\mathbb R^2}[\delta u_\varepsilon\cdot\nabla q_e+(\bar Q+\alpha)(\nabla\cdot\delta v_\varepsilon+\nabla\cdot v\chi_{\mathcal O}(x,y,t))]\delta q_{e\varepsilon} dxdy,\label{newadd}
\end{align}
a.e.~$t\in(0,\infty)$.
Recalling that $q_e\leq0$, we have
\begin{equation}\label{con3}
\int_{\mathbb R^2}q_{e\varepsilon}^+\delta q_{e\varepsilon}dxdy= \int_{\mathbb R^2}q_{e\varepsilon}^+(q_{e\varepsilon}-q_e)dxdy\geq \int_{\mathbb R^2}q_{e\varepsilon}^+q_{e\varepsilon}dxdy=\|q_{e\varepsilon}^+\|_2^2.
\end{equation}
Note that $\partial_tq_e+u\cdot\nabla q_e=0$, a.e.~on $\mathcal O$, it follows from (\ref{ineq5}) that $\nabla\cdot v\leq0$, a.e.~on $\mathcal O$, and thus
\begin{align*}
  &-\nabla\cdot v\chi_{\mathcal O}(x,y,t)\delta q_{e\varepsilon}=-\nabla\cdot v\chi_{\mathcal O}(x,y,t)q_{e\varepsilon}\leq-\nabla\cdot v\chi_{\mathcal O}(x,y,t)q_{e\varepsilon}^+.
\end{align*}
Thanks to the above inequality, it follows from (\ref{newadd}), (\ref{con3}) and the Young inequality that
\begin{align*}
  &\frac12\frac{d}{dt}\|\delta q_{e\varepsilon}\|_2^2+\frac{1+\alpha}{\varepsilon}\| q_{e\varepsilon}^+\|_2^2\\
  \leq&-\int_{\mathbb R^2} [\delta u_\varepsilon\cdot\nabla q_e +(\bar Q+\alpha)\nabla\cdot\delta v_\varepsilon]\delta q_{e\varepsilon}dxdy-(\bar Q+\alpha)\int_{\mathbb R^2}\nabla\cdot v\chi_{\mathcal O}(x,y,t)q_{e\varepsilon}^+ dxdy\\
  \leq&\int_{\mathbb R^2}|\nabla q_e||\delta u_\varepsilon||\delta q_{e\varepsilon}|dxdy+\frac14\|\nabla\delta v_\varepsilon\|_2^2+(\bar Q+\alpha)^2\|\delta q_{e\varepsilon}\|_2^2\\
  &+\frac{1+\alpha}{2\varepsilon}\|q_{e\varepsilon}^+ \|_2^2+\frac{(\bar Q+\alpha)^2}{2(1+\alpha)}\varepsilon\|\nabla v\|_2^2,
\end{align*}
and thus
\begin{align}
  \frac{d}{dt}\|\delta q_{e\varepsilon}\|_2^2+\frac{1+\alpha}{\varepsilon}\|q_{e\varepsilon}^+\|_2^2
  \leq&\frac12\|\nabla\delta v_\varepsilon\|_2^2+2(\bar Q+\alpha)^2\|\delta q_{e\varepsilon}\|_2^2+\frac{(\bar Q+\alpha)^2}{1+\alpha}\varepsilon\|\nabla v\|_2^2\nonumber\\
  &+2\int_{\mathbb R^2}|\nabla q_e||\delta u_\varepsilon||\delta q_{e\varepsilon}|dxdy.\label{con4}
\end{align}

Summing (\ref{con1}) with (\ref{con4}) yields
\begin{align*}
  \frac{d}{dt}\|(\delta u_\varepsilon,&\delta v_\varepsilon,\delta T_{e\varepsilon},\delta q_{e\varepsilon})\|_2^2+\frac12\|(\nabla\delta u_\varepsilon,\nabla\delta v_\varepsilon)\|_2^2 +\frac{1+\alpha}{\varepsilon}\|q_{e\varepsilon}^+\|_2^2\nonumber\\
  \leq&C\int_{\mathbb R^2}[(|\nabla u|+|\nabla v|+|v|^2)(|\delta u_\varepsilon|^2+|\delta v_\varepsilon|^2)+|\delta T_{e\varepsilon}|^2+|\delta q_{e\varepsilon}|^2\nonumber\\
  &+|\nabla T_e||\delta u_\varepsilon||\delta T_{e\varepsilon}|+|\nabla q_e||\delta u_\varepsilon||\delta q_{e\varepsilon}|]dxdy+C\varepsilon\|\nabla v\|_2^2,
\end{align*}
from which, it follows from the H\"older, Ladyzhenskay, Gagliardo-Nirenberg, $\|\varphi\|_{\frac{2m}{m-2}}\leq C\|\varphi\|_2^{\frac{m-2}{m}}\|\nabla\varphi\|_2^{\frac2m}$, and Young inequalities that
\begin{align}
&\frac{d}{dt}\|(\delta u_\varepsilon,\delta v_\varepsilon, \delta T_{e\varepsilon},\delta q_{e\varepsilon})\|_2^2 +\frac12\|(\nabla\delta u_\varepsilon,\nabla\delta v_\varepsilon)\|_2^2 +\frac{1+\alpha}{\varepsilon}\|q_{e\varepsilon}^+\|_2^2\nonumber\\
  \leq&C\left(\|(\nabla u , \nabla v )\|_2 +\| v \|_4^2\right)\|(\delta u_\varepsilon,\delta v_\varepsilon)\|_4^2+C\|(\delta T_{e\varepsilon}, \delta q_{e\varepsilon})\|_2^2\nonumber\\
  &+C\|(\nabla T_e,\nabla q_e)\|_m \|\delta u_\varepsilon\|_{\frac{2m}{m-2}}\|(\delta T_{e\varepsilon},  \delta q_{e\varepsilon})\|_2+C\varepsilon\|\nabla v\|_2^2\nonumber\\
  \leq&C\left(\|(\nabla u, \nabla v)\|_2+\|v\|_2\|\nabla v\|_2\right)\|(\delta u_\varepsilon,\delta v_\varepsilon)\|_2  \|(\nabla\delta u_\varepsilon,\nabla\delta v_\varepsilon)\|_2+C\varepsilon\|\nabla v\|_2^2\nonumber\\
  &+C\|(\delta T_{e\varepsilon}, \delta q_{e\varepsilon})\|_2^2+C\|(\nabla T_e,\nabla  q_e)\|_m\|\delta u_\varepsilon\|_{2}^{\frac{m-2}{m}}\|\nabla\delta u_\varepsilon\|_2^{\frac2m}\|(\delta T_{e\varepsilon}, \delta q_{e\varepsilon})\|_2\nonumber\\
  \leq&\frac14\|(\nabla\delta u_\varepsilon,\nabla\delta v_\varepsilon)\|_2^2+C\left(\|(\nabla u, \nabla v )\|_2^2+\|v \|_2^2\|\nabla v\|_2^2\right)\|(\delta u_\varepsilon,\delta v_\varepsilon)\|_2^2 +C\varepsilon\|\nabla v\|_2^2\nonumber\\
  &+C\|(\delta T_{e\varepsilon}, \delta q_{e\varepsilon})\|_2^2+C\|(\nabla T_e,\nabla q_e)\|_m^{\frac{m}{m-1}}\|\delta u_\varepsilon\|_{2}^{\frac{m-2}{m-1}} \|(\delta T_{e\varepsilon}, \delta q_{e\varepsilon})\|_2^{\frac{m}{m-1}}\nonumber\\
  \leq&\frac14\|(\nabla\delta u_\varepsilon,\nabla\delta v_\varepsilon)\|_2^2+C\Big(1+\|(\nabla u, \nabla v)\|_2^2
  +\|v\|_2^2\|\nabla v\|_2^2\nonumber\\
  &+\|(\nabla  T_e,\nabla q_e)\|_m^{\frac{m}{m-1}}\Big)
  \|(\delta u_\varepsilon,\delta v_\varepsilon,\delta T_{e\varepsilon}, \delta q_{e\varepsilon})\|_2^2+C\varepsilon\|\nabla v\|_2^2.
\end{align}
Therefore, we have
\begin{align*}
  \frac{d}{dt}\|(\delta u_\varepsilon,&\delta v_\varepsilon,\delta T_{e\varepsilon},\delta q_{e\varepsilon})\|_2^2+\frac14\|(\nabla\delta u_\varepsilon,\nabla\delta v_\varepsilon)\|_2^2 +\frac{1+\alpha}{\varepsilon}\|q_{e\varepsilon}^+\|_2^2\nonumber\\
  \leq&C\Big(1+\|(\nabla u, \nabla v)\|_2^2
  +\|v\|_2^2\|\nabla v\|_2^2
  +\|(\nabla  T_e,\nabla q_e)\|_m^{\frac{m}{m-1}}\Big)\\
  &\times\|(\delta u_\varepsilon,\delta v_\varepsilon,\delta T_{e\varepsilon}, \delta q_{e\varepsilon})\|_2^2+C\varepsilon\|\nabla v\|_2^2.
\end{align*}

Applying the Gronwall inequality to the above inequality and recalling the regularities of $(u, v, T_e, q_e)$ yield
\begin{align*}
  &\sup_{0\leq t\leq\mathcal T}\|(\delta u_\varepsilon,\delta v_\varepsilon,\delta T_{e\varepsilon},\delta q_{e\varepsilon})(t)\|_2^2+\int_0^\mathcal T\left(\|(\nabla\delta u_\varepsilon,\nabla\delta v_\varepsilon)\|_2^2 +\frac{\|q_{e\varepsilon}^+\|_2^2}{\varepsilon}\right)dt\leq C\varepsilon,
\end{align*}
for a positive constant $C$ depending only on
$\alpha,\bar Q, \mathcal T, m$, $\|(u_0, v_0, T_{e,0}, q_{e,0})\|_{H^1}$ and $\|(\nabla T_{e,0}, \nabla q_{e,0})\|_{m}$. This proves the desired estimate in the theorem, while the strong convergences are direct consequences of this estimate.
\end{proof}

\section{Appendix}

In this appendix, we state and prove several parabolic estimates, which have been used in the previous sections.

\begin{lemma}\label{lemapp1}
Given a time $\mathcal T\in(0,\infty)$, and a function $g\in L^\alpha(0,\mathcal T; L^\beta(\mathbb R^2))$, with $1<\alpha,\beta<\infty$. Let $U$ be the unique solution to
\begin{equation*}
  \left\{
  \begin{array}{l}
  \partial_t U-\Delta U=g,\quad\mbox{in }\mathbb R^2\times(0,\mathcal T), \\
  U|_{t=0}=0,\quad\mbox{in }\mathbb R^2.
  \end{array}
  \right.
\end{equation*}
Then, we have the estimate
$$
\|\partial_tU\|_{L^\alpha(0,\mathcal T; L^\beta(\mathbb R^2))}+\|\Delta U\|_{L^\alpha(0,\mathcal T; L^\beta(\mathbb R^2))}\leq C_{\alpha,\beta}\|g\|_{L^\alpha(0,\mathcal T; L^\beta(\mathbb R^2))},
$$
where $C_{\alpha,\beta}$ is a positive constant depending only on $\alpha,\beta$, and in particular is independent of $\mathcal T$ and $g$.
\end{lemma}

\begin{proof}
  Introducing the scaled functions $U_{\mathcal T}$ and $g_{\mathcal T}$ as
  $$
  U_{\mathcal T}(x,t)=U(\sqrt{\mathcal T}x,\mathcal Tt), \quad g_{\mathcal T}(x,t)=g(\sqrt{\mathcal T}x,\mathcal Tt),\quad x\in\mathbb R^2, t\in(0,1),
  $$
  then one can easily verify that $U_{\mathcal T}$ and $g_{\mathcal T}$ satisfy
  \begin{equation*}
  \left\{
  \begin{array}{l}
  \partial_t U_{\mathcal T}-\Delta U_{\mathcal T}=\mathcal Tg_{\mathcal T}, \quad\mbox{in }\mathbb R^2\times(0,1), \\
  U|_{t=0}=0,\quad\mbox{in }\mathbb R^2.
  \end{array}
  \right.
\end{equation*}
 Applying the maximal regularity theory for parabolic equations to the above system  (see, e.g., \cite{COULHONDUONG},  \cite{HIEBERPRUSS} and \cite{KUNWEI}),  one has
$$
\|\partial_tU_\mathcal T\|_{L^\alpha(0,1; L^\beta(\mathbb R^2))}+\|\Delta U_\mathcal T\|_{L^\alpha(0,1; L^\beta(\mathbb R^2))}\leq C_{\alpha,\beta}\mathcal T\|g_\mathcal T\|_{L^\alpha(0,\mathcal T; L^\beta(\mathbb R^2))}.
$$
From which, and after observing that,
\begin{eqnarray*}
  &&\|\partial_tU_{\mathcal T}\|_{L^\alpha(0,1; L^\beta(\mathbb R^2))}=\mathcal T^{1-\frac 1\alpha-\frac1\beta}\|\partial_tU\|_{L^\alpha(0,\mathcal T; L^\beta(\mathbb R^2))},\\
  &&\|\Delta U_{\mathcal T}\|_{L^\alpha(0,1; L^\beta(\mathbb R^2))}=\mathcal T^{1-\frac 1\alpha-\frac1\beta}\|\Delta U\|_{L^\alpha(0,\mathcal T; L^\beta(\mathbb R^2))},\\
  &&\|g_{\mathcal T}\|_{L^\alpha(0,1; L^\beta(\mathbb R^2))}=\mathcal T^{-\frac 1\alpha-\frac1\beta}\|g\|_{L^\alpha(0,\mathcal T; L^\beta(\mathbb R^2))},
\end{eqnarray*}
one obtains the conclusion.
\end{proof}

\begin{lemma}\label{lemapp2}
Given a time $\mathcal T\in(0,\infty)$, and let $f$ and $g$ be two functions, such that $f\in L^2(\mathbb R^2\times(0,\mathcal T))$ and $g\in L^4(\mathbb R^2\times(0,\mathcal T)$. Let $v$ be the unique solution to
\begin{equation*}
  \left\{
  \begin{array}{l}
  \partial_tv-\Delta v=f+\nabla g,\quad\mbox{in }\mathbb R^2\times(0,\mathcal T), \\
  v|_{t=0}=v_0\in H^1(\mathbb R^2),\quad\mbox{in }\mathbb R^2.
  \end{array}
  \right.
\end{equation*}
Then we have the following estimate
$$
\int_0^{\mathcal T}\|\nabla v\|_4^4dt\leq C\left(\|\nabla v_0\|_2^4+\left(\int_0^\mathcal T\|f\|_2^2dt\right)^2+\int_0^\mathcal T\|g\|_4^4dt\right),
$$
where $C$ is an absolute constant, and in particular is independent of $\mathcal T, v_0$, $f$ and $g$.
\end{lemma}

\begin{proof}
  Decompose $v$ as $v=\bar v+\hat v$, where $\bar v$ and $\hat v$ are the unique solutions to systems
  \begin{equation*}
  \left\{
  \begin{array}{l}
  \partial_t\bar v-\Delta\bar  v=f, \quad\mbox{in }\mathbb R^2\times(0,\mathcal T), \\
  \bar v|_{t=0}=v_0\in H^1(\mathbb R^2),\quad\mbox{in }\mathbb R^2,
  \end{array}
  \right.
\end{equation*}
and
\begin{equation}\label{App1}
  \left\{
  \begin{array}{l}
  \partial_t\hat v-\Delta \hat v=\nabla g,\quad\mbox{in }\mathbb R^2\times(0,\mathcal T), \\
  v|_{t=0}=0,\quad\mbox{in }\mathbb R^2,
  \end{array}
  \right.
\end{equation}
respectively.
The standard energy approach (multiplying the equation for $\bar v$ by $-\Delta\bar v$, integrating over $\mathbb R^2$, integration by parts, using the Young, and integrating with respect to $t$ over $(0,\mathcal T)$) to the system for $\bar v$ leads to
$$
\sup_{0\leq t\leq\mathcal T}\|\nabla\bar v(t)\|_2^2+\int_0^\mathcal T\|\Delta\bar v\|_2^2dt\leq\|\nabla v_0\|_2^2+\int_0^\mathcal T\|f\|_2^2dt.
$$

Defining $U$ to be the unique solution to the system
\begin{equation*}
  \left\{
  \begin{array}{l}
  \partial_t U-\Delta U=g,\quad\mbox{in }\mathbb R^2\times(0,\mathcal T), \\
  U|_{t=0}=0,\quad\mbox{in }\mathbb R^2.
  \end{array}
  \right.
\end{equation*}
Then $\nabla U$ satisfies the same system as that for $\hat v$, and therefore, by the uniqueness of the solutions to system (\ref{App1}), we have $\hat v=\nabla U$. Thanks to this fact, and applying Lemma \ref{lemapp1}, it follows from the elliptic estimates that
\begin{eqnarray*}
&\|\nabla\hat v\|_{L^4(0,\mathcal T; L^4(\mathbb R^2))}=\|\nabla^2U\|_{L^4(0,\mathcal T; L^4(\mathbb R^2))}\\
&\leq C\|\Delta U\|_{L^4(0,\mathcal T; L^4(\mathbb R^2))}\leq C\|g\|_{L^4(0,\mathcal T; L^4(\mathbb R^2))},
\end{eqnarray*}
for an absolute positive constant $C$.

Combining the estimates for $\bar v$ and $\hat v$, we deduce from the Ladyzhenskaya inequality that
\begin{align*}
  \int_0^\mathcal T\|\nabla v\|_4^4dt\leq&C\int_0^\mathcal T\|\nabla\bar v\|_4^4dt+C\int_0^\mathcal T\|\nabla\hat v\|_4^4dt\\
  \leq&C\left(\sup_{0\leq t\leq \mathcal T}\|\nabla\bar v(t)\|_2^2\right)\int_0^\mathcal T\|\Delta\bar v\|_2^2dt+ C\|g\|_{L^4(0,\mathcal T; L^4(\mathbb R^2))}^4\\
  \leq&C\left(\|\nabla v_0\|_2^4+\left(\int_0^\mathcal T\|f\|_2^2dt\right)^2+\int_0^\mathcal T\|g\|_4^4dt\right),
\end{align*}
for an absolute positive constant $C$. This completes the proof.
\end{proof}

\begin{lemma}
\label{lemapp3}
Given a time $\mathcal T\in(0,\infty)$ and a number $m\in(2,\infty)$. Let $f\in L^2(0,\mathcal T; L^m(\mathbb R^2))$, and $v$ be the unique solution to
\begin{equation*}
  \left\{
  \begin{array}{l}
  \partial_tv-\Delta v=f,\quad\mbox{in }\mathbb R^2\times(0,\mathcal T),\\
  v|_{t=0}=v_0\in H^1(\mathbb R^2).
  \end{array}
  \right.
\end{equation*}
Then, we have the following estimate
$$
\int_0^\mathcal T\|\Delta v\|_mdt\leq C_m(1+\sqrt\mathcal T)\left[\|\nabla v_0\|_2+ \left(\int_0^\mathcal T\|f\|_m^2dt\right)^{\frac12}\right],
$$
where $C_m$ is a positive constant depending only on $m$, and in particular is independent of $\mathcal T, f$ and $v_0$.
\end{lemma}

\begin{proof}
Decompose $v$ as $v=\bar v+\hat v$, where $\bar v$ and $\hat v$ are the unique solutions to systems
\begin{equation*}
  \left\{
  \begin{array}{l}
  \partial_t\bar v-\Delta \bar v=f,\quad\mbox{in }\mathbb R^2\times(0,\mathcal T),\\
  v|_{t=0}=0,
  \end{array}
  \right.
\end{equation*}
and
\begin{equation}
  \left\{
  \begin{array}{l}
  \partial_t\hat v-\Delta\hat v=0,\quad\mbox{in }\mathbb R^2\times(0,\mathcal T),\\
  v|_{t=0}=v_0\in H^1(\mathbb R^2),\label{App2}
  \end{array}
  \right.
\end{equation}
respectively.

By Lemma \ref{lemapp1} and using the H\"older inequality, for $\bar v$, we have the estimate
$$
\int_0^\mathcal T\|\Delta\bar v\|_mdt\leq\mathcal T^{\frac12}\left(\int_0^\mathcal T\|\Delta\bar v\|_m^2dt\right)^{\frac12}\leq C_m\mathcal T^{\frac12}\left(\int_0^\mathcal T\|f\|_m^2dt\right)^{\frac12}.
$$
To estimate $\hat v$, we multiplying equation (\ref{App2}) by $t\Delta^2\hat v-\Delta\hat v$, integrating the resultant over $\mathbb R^2$, then it follows from integration by parts that
$$
\frac12\frac{d}{dt}(\|\nabla\hat v\|_2^2+\|\sqrt t\Delta\hat v\|_2^2)+\frac12\|\Delta\hat v\|_2^2+\|\sqrt t\nabla\Delta\hat v\|_2^2=0,
$$
from which, integrating with respect to $t$ yields
$$
\sup_{0\leq t\leq\mathcal T}(\|\nabla\hat v(t)\|_2^2+\|\sqrt t\Delta\hat v(t)\|_2^2)
+\int_0^\mathcal T(\|\Delta\hat v\|_2^2+\|\sqrt t\nabla\Delta\hat v\|_2^2)dt\leq\|\nabla v_0\|_2^2.
$$
Thanks to this estimate, by the Gagliardo-Nirenberg, $\|\varphi\|_m\leq C\|\varphi\|_2^{\frac2m}\|\nabla\varphi\|_2^{1-\frac2m}$, and H\"older inequalities, we deduce
\begin{align*}
  \int_0^\mathcal T&\|\Delta\hat v\|_mdt\leq C\int_0^\mathcal T\|\Delta\hat v\|_2^{\frac2m}\|\nabla\Delta\hat v\|_2^{1-\frac2m}dt\\
  =&C\int_0^\mathcal T\|\Delta\hat v\|_2^{\frac2m}\|\sqrt t\nabla\Delta\hat v\|_2^{1-\frac2m}t^{-\frac12(1-\frac2m)}dt\\
  \leq&C\left(\int_0^\mathcal T\|\Delta\hat v\|_2^2dt\right)^{\frac1m}\left(\int_0^\mathcal T\|\sqrt t\nabla\Delta\hat v\|_2^2dt\right)^{\frac{m-2}{2m}}\left(\int_0^\mathcal Tt^{-(1-\frac2m)}dt\right)^{\frac12}\\
  \leq&C\sqrt m\mathcal T^{\frac1m}\|\nabla v_0\|_2.
\end{align*}

Combining the estimates for $\bar v$ and $\hat v$, we then deduce from the Young inequality (recalling $m>2$) that
\begin{align*}
  \int_0^\mathcal T&\|\Delta v\|_mdt\leq\int_0^\mathcal T(\|\Delta\hat v\|_m+\|\Delta\bar v\|_m)dt\\
  \leq&C_m\mathcal T^{\frac12}\left(\int_0^\mathcal T\|f\|_m^2dt\right)^{\frac12}+C\sqrt m\mathcal T^{\frac1m}\|\nabla v_0\|_2\\
  \leq&C_m(1+\sqrt\mathcal T)\left[\|\nabla v_0\|_2+ \left(\int_0^\mathcal T\|f\|_m^2dt\right)^{\frac12}\right],
\end{align*}
proving the conclusion.
\end{proof}
\section*{Acknowledgments}
{E.S.T. is thankful to the kind hospitality of the \'Ecole Polytechnique (CMLS), Paris, where part of this work was completed. This work was supported in part by the ONR grant N00014-15-1-2333 and the NSF grants DMS-1109640 and DMS-1109645.
}
\par

\end{document}